%% file: SIAM-Unconditional_MBP_and_energy_dissipation_L1_scheme_for_TFAC.tex
\documentclass[onefignum,onetabnum]{siamart190516}
\usepackage{comment}
\newcommand{\dps}{\displaystyle}

\def\bq{\begin{equation}}
	\def\eq{\end{equation}}
\def\beq{\begin{equation*}}
	\def\eeq{\end{equation*}}
\def\brr{\bq\begin{array}{r@{}l}}
	\def\err{\end{array}\eq}
\def\bry{\beq\begin{array}{r@{}l}}
	\def\ery{\end{array}\eeq}
\def\vPhi{{\vec{\Phi}}}

\def\ve{{\vec{e}}}
\font\tenbi=cmmib10   at 11 pt
\font\sevenbi=cmmib10 at 9pt
\font\fivebi=cmmib7 at 6pt
\newfam\bifam
\textfont\bifam=\tenbi \scriptfont\bifam=\sevenbi  \scriptscriptfont\bifam=\fivebi
\def\bi{\fam\bifam\tenbi}
\def\x{{\bi x}}
\def\n{{\bi n}}

\def\R{\mathbb{R}}

\def\Dt {\tau}

\input{ex_shared}
\usepackage{subfigure}
\allowdisplaybreaks[4]
\ifpdf
\hypersetup{
  pdftitle={A linear unconditionally structure-preserving L1 scheme for the time-fractional Allen-Cahn equation},
  pdfauthor={Dianming Hou, Zhonghua Qiao, and Tao Tang}
}
\fi

\begin{document}
\graphicspath{{figures/},}
\maketitle

\begin{abstract}
 As a variational phase-field model, the time-fractional Allen–Cahn (TFAC) equation enjoys the maximum bound principle (MBP) and a variational energy dissipation law. In this work, we develop and analyze linear, structure-preserving time-stepping schemes for TFAC, including first-order and $\min\{1+\alpha, 2-\alpha\}$-order L1 discretizations, together with fast implementations based on the sum-of-exponentials (SOE) technique. A central feature of the proposed linear schemes is their unconditional preservation of both the discrete MBP and the variational energy dissipation law on general temporal meshes, including graded meshes commonly used for these problems. Leveraging the MBP of the numerical solutions, we establish sharp error estimates by employing the time-fractional Gro\"nwall inequality. Finally, numerical experiments validate the theoretical results and demonstrate the effectiveness of the proposed schemes with an adaptive time-stepping strategy.
\end{abstract}


\section{Introduction}
\setcounter{equation}{0}

In this paper, we focus on numerical investigation on the following time-fractional Allen-Cahn (TFAC) model, which can be written as:
\begin{align}\label{TFAC}
_{0}{}\!D^{\alpha}_{t}\phi-\varepsilon^{2}\Delta\phi+f(\phi)=0,\quad(\x,t)\in\Omega\times(0,T],
\end{align}
a variant of gradient flows that incorporates a modified dissipation mechanism.
Here, $\Omega$ is a bounded Lipschitz domain in $\mathbb{R}^{d} (d = 1,2,3)$, $T>0$ is the
terminal time, $\phi(\x,t)$ denotes the state (or phase) function, $f(\phi)$ is a nonlinear potential function, $\varepsilon>0$ is related to the diffuse interface width and $_{0}{}\!D^{\alpha}_{t}$ is the Caputo fractional derivative defined by
\begin{align}\label{frac_d}
_{0}{}\!D^{\alpha}_{t}\phi(t)={}_{0}{}\!I^{1-\alpha}_{t}\phi'(t),\quad 0<\alpha<1
\end{align}
with the frational Riemann-Liouville integral operator $_{0}{}\!I^{\alpha}_{t}$ given by
\begin{align}
_{0}{}\!I^{\alpha}_{t}\phi(t)=\frac{1}{\Gamma(\alpha)}\int_{0}^{t}(t-s)^{\alpha-1}\phi(s)ds.
\end{align}
The above TFAC model \eqref{TFAC} satisfies a variational energy dissipation law \cite{LTZ21}:
\bq\label{Eg_law}
\frac{d\mathcal{E}_{\alpha}[\phi]}{dt}\leq-\frac{t^{\alpha-1}}{2\Gamma(\alpha)} \Big\| \frac{\delta E[\phi]}{\delta \phi}\Big\|^{2}\leq 0,
\eq
where the variational energy $\mathcal{E}_{\alpha}[\phi]$ is defined by
\bq\label{eqq1}
\mathcal{E}_{\alpha}[\phi]:=E[\phi]+\frac{1}{2}{}_{0}\!I^{\alpha}_{t}\Big\| \frac{\delta E[\phi]}{\delta \phi}\Big\|^{2},
\eq
and the free energy $E[\phi]$ is given by
$$E[\phi]=\int_{\Omega}\frac{\varepsilon^{2}}{2}|\nabla\phi|^{2}+F(\phi)d\x$$
with $F'(\phi)=f(\phi).$
Another intrinsic property of the TFAC equation \eqref{TFAC} is the maximum bound principle (MBP), which guarantees the physical admissibility of the phase variable: if the initial data satisfy $|\phi(\x,0)| \leq 1$ for all $\x \in \Omega$, then the solution remains strictly bounded within $[-1,1]$ for all $\x \in \Omega$ and $t \geq 0$; see \cite{TYZ18} for further discussions. These two intrinsic properties are essential for the thermodynamic consistency and physical reliability of the model, and thus preserving them at the discrete level is a key challenge in numerical analysis.

Fractional gradient flows have drawn growing interest recently due to their inherent memory effects and nonlocality. A fractional derivative, defined as a weighted average over the entire past history, links the current rate of change to accumulated past rates and provides a natural framework for modeling viscoelastic materials, polymers, and other complex media. Key early motivations and results are collected in surveys and foundational works such as \cite{ASS16,DJLQ18,LWY17,LCWZ18,QW22,SXK16,ZLWW14}.
Numerical strategies for time-fractional phase-field models have diversified significantly. For instance, Tang et al. \cite{TYZ18} established an integral-type energy dissipation law and analyzed the L1 time-discretization for time-fractional derivative on uniform meshes, achieving first-order accuracy in an energy--stable sense. Du et al. \cite{DYZ20} designed convex--splitting and weighted-stabilization schemes and proved $\alpha$-order convergence on uniform meshes without imposing additional smoothness assumptions on the solution. More recently, Liao et al. \cite{JLGZ20} developed an adaptive second--order Crank--Nicolson method via the scalar auxiliary variable (SAV) framework for the time-fractional molecular beam epitaxy (MBE) model, and demonstrated its unconditional stability on nonuniform meshes.
The notion of a nonlocal energy dissipation law for time-fractional phase-field models was first derived in \cite{HX21} for arbitrary temporal meshes. Building on this, linear high-order and modified nonlocal energy-dissipation schemes were developed for TFAC and time-fractional MBE using the SAV approach \cite{HQT23,HX21, HX22}.
Further progress includes first-order stabilized semi-implicit L1 schemes with discrete fractional and weighted energy laws \cite{QTY25}, and and extensions to a $(2-\alpha)$-order L1-SAV scheme with energy stability under nonuniform time discretization. 
A $(1+\alpha)$-order nonlinear scheme with variable steps based on Crank–Nicolson was studied for TFAC, preserving the discrete variational energy law \eqref{Eg_law}   unconditionally while enforcing a time-step restriction to guarantee MBP \cite{LTZ21}.
 A subsequent second-order nonlinear scheme for time-fractional phase-field equations preserved the discrete energy law \eqref{Eg_law} under a step-ratio constraint \cite{LLL23}.
A common challenge across these nonlinear schemes is the need to solve nonlinear algebraic systems at each time step, accompanied by time-step restrictions to ensure solvability and stability. This motivates the development of efficient high-order linear, structure-preserving schemes for TFAC.
In this line, \cite{JLZ20} introduced two linear, first-order, fast L1 time-stepping schemes (backward Euler and stabilized semi-implicit) for TFAC and proved MBP preservation. However, whether these linear schemes satisfy a discrete energy-dissipation law remained unresolved, leaving an open question and a prompt for further investigation. Overall, the field highlights a pressing need for linear, high-order, structure-preserving methods that maintain MBP and energy-dissipation properties on general (including nonuniform or adaptive) temporal meshes, while avoiding nonlinear solves at each step. 

The objective of this paper is to develop and analyze linear high-order schemes for the TFAC equation that unconditionally preserve the MBP and the variant energy dissipation law \eqref{Eg_law}. The key idea is to combine a linear stabilization technique with a predictor-corrector strategy based on the L1 time-stepping approximation of the fractional derivative.
The main contributions of the paper are summarized as follows:
\begin{itemize}
\item We construct $\min\{1+\alpha, 2-\alpha\}$-order linear predictor-corrector L1 stabilized schemes for the TFAC equation, which unconditionally preserve both the discrete MBP and the discrete variant of the energy dissipation law \eqref{Eg_law}. The schemes are fully linear and easy to implement, requiring the solution of only two Poisson-type equations with constant coefficients at each time step.

\item We rigorously establish the discrete preservation of the variant energy dissipation law \eqref{Eg_law} for both the first-order and the proposed $\min\{1+\alpha, 2-\alpha\}$-order L1 stabilized schemes.

\item We extend the fractional discrete Gr\"onwall inequality in Theorem 3.1 of \cite{LMZ19} to the case $\theta=1$. This extension plays a crucial role in eliminating the time-step size restriction in the error analysis of the proposed linear schemes, a limitation that inevitably arises when employing the standard fractional discrete Gr\"onwall inequality with $0 \leq \theta < 1$ in \cite{LMZ19}.
\end{itemize}

The remainder of the paper is organized as follows. In Section 2, we construct the first-order and the $\min\{1+\alpha, 2-\alpha\}$-order linear L1 stabilized schemes for the TFAC equation, and establish their unconditional preservation of the MBP and the variant energy dissipation law on general time grids. Section 3 provides a rigorous convergence analysis of the fast linear predictor-corrector stabilized L1 scheme \eqref{fL1_h} for the TFAC problem \eqref{TFAC}. Section 4 presents a series of numerical experiments that validate the theoretical results and demonstrate the efficiency of the proposed schemes. Finally, concluding remarks are given in Section 5.

\section {Fully discrete linear stabilized L1 schemes}\label{sect2a}
\setcounter{equation}{0}
Without loss of generality, we focus on the two-dimensional case ($d=2$) of \eqref{TFAC} subject to the homogeneous Neumann boundary condition, i.e., $\frac{\partial \phi}{\partial \n}\big|_{\partial \Omega}=0$.
The proposed scheme and the associated theoretical results can be readily extended to higher dimensions and to other types of boundary conditions.
\subsection{Spatial discretization by central difference}
We use the notations and preliminary results of the central difference function spaces and operators reported in~\cite{BHLWWZ13,BLWW13,HJQ22,LSR19,SWWW12,WISE10}. For more complete details, one can refer to these works.
For simplicity, we consider a square computational domain $\Omega=(0,L)\times(0,L),$ and a uniform spatial grid spacing $h=L/M$.
Let $E_{M}$ and $C_{M}$ denote two types of point sets defined at the space nodes and at the midpoints of the space nodes, respectively. That is
$$\mathbf{E}_{M}=\Big\{x_{i+\frac{1}{2}}=ih\;\Big|\; i=0,1, \cdots,M\Big\},\qquad \mathbf{C}_{M}=\Big\{x_{i}=\Big(i-\frac{1}{2}\Big)h\;\Big|\; i=1, \cdots,M\Big\}.$$
Then, we define the following discrete grid function spaces:
\begin{align*}
& \mathcal{C}_{h}=\dps\{U: \mathbf{C}_{M}\times\mathbf{C}_{M}\rightarrow\R\;\big|\;U_{i,j}, \;1\leq i,j\leq M\},\\
 & e^{x}_{h}=\dps\{U: \mathbf{E}_{M}\times\mathbf{C}_{M}\rightarrow\R\;\big|\;U_{i+\frac{1}{2},j}, \;0\leq i\leq M,\;1\leq j\leq M\},\\
 & e^{y}_{h}=\dps\{U: \mathbf{C}_{M}\times\mathbf{E}_{M}\rightarrow\R\;\big|\;U_{i,j+\frac{1}{2}}, \;1\leq i\leq M,\;0\leq j\leq M\},\\
 & e^{x}_{0,h}=\dps\{U\in e^{x}_{h}\;\big|\;U_{\frac{1}{2},j}=U_{M+\frac{1}{2},j}=0,\; 1\leq j\leq M\},\\
 & e^{y}_{0,h}=\dps\{U\in  e^{y}_{h}\;\big|\;U_{i,\frac{1}{2}}=U_{i,M+\frac{1}{2}}=0,\; 1\leq i\leq M\}.
\end{align*}
Next, we will define discrete gradient and divergence operators, which are compatible with the homogeneous Neumann boundary condition.
For any $U\in \mathcal{C}_h$, the discrete gradient operator $\nabla_{h} =(\nabla^x_{h}, \nabla^y_{h}): \mathcal{C}_h\rightarrow( e_{0,h}^{x},  e^{y}_{0,h})$ is given by
\begin{align*}
(\nabla^x_{h}U)_{i+\frac{1}{2},j}=\dps\frac{U_{i+1,j}-U_{i, j}}{h}, \qquad 1\leq i\leq M-1, 1\leq j\leq M,\\
(\nabla^y_{h}U)_{i,j+\frac{1}{2}}=\dps\frac{U_{i,j+1}-U_{i, j}}{h}, \qquad 1\leq i\leq M, 1\leq j\leq M-1.
\end{align*}
 In addition, the discrete divergence operator $\nabla_{h}\cdot:( e_{h}^{x},  e^{y}_{h})\rightarrow\mathcal{C}_h$ is defined by
\begin{align*}
(\nabla_{h}\cdot (U^{x},U^{y})^T)_{i,j}= \frac{U^x_{i+\frac{1}{2},j}-U^x_{i-\frac{1}{2},j}}{h}+ \frac{U^y_{i,j+\frac{1}{2}}-U^y_{i,j-\frac{1}{2}}}{h},\quad 1\leq i,j\leq M,
\end{align*}
for any $(U^{x},U^{y})^{T}\in ( e_{h}^{x},  e^{y}_{h}).$
Then, we obtain the discrete Laplacian $\Delta_{h}:\mathcal{C}_{h}\rightarrow\mathcal{C}_{h}$:
\begin{align*}
(\Delta_{h}U)_{i,j}=(\nabla_{h}\cdot(\nabla_{h} U))_{i,j},\quad 1\leq i,j\leq M.
\end{align*}
The corresponding discrete inner-products are defined as follows:
\begin{align*}
&
\big\langle U,V\big\rangle_{\Omega}=\dps h^{2}\sum_{i,j=1}^{M}U_{i,j}V_{i,j}, &&  \forall U,V\in\mathcal{C}_{h},\\
&[U^x,V^x]_{x}=\big\langle a_{x}(U^xV^x),1\big\rangle _{\Omega},&&  \forall U^x,V^x\in e^{x}_{h},\\
&[U^y,V^y]_{y}=\dps\big\langle a_{y}(U^yV^y),1\big\rangle _{\Omega},&&  \forall U^y,V^y\in e^{y}_{h},\\
& [(U^{x},U^{y})^{T},(V^{x},V^{y})^{T}]_{\Omega}=[U^{x},V^{x}]_{x}+[U^{y},V^{y}]_{y},
\end{align*}
where $a_{x}: e^{x}_{h}\rightarrow\mathcal{C}_h$ and $a_{y}: e^{y}_{h}\rightarrow\mathcal{C}_h$ are the two average operators
 $$(a_{x}U)_{i,j}=({U_{i+1/2,j}+U_{i-1/2,j}})/2\quad\text{
 and}\quad (a_{y}U)_{i,j}=({U_{i,j+1/2}+U_{i,j-1/2}})/2$$ for $1\leq i,j\leq M$, respectively.
Then,  for any $U\in\mathcal{C}_{h}$, its corresponding discrete  $L^2, H^1$ semi-norms and norms, and the $L^{\infty}$-norm are respectively given by:
\begin{align*}
&\dps\|U\|^{2}_{h}=\big\langle U,U\big\rangle _{\Omega},&& \|\nabla_{h}U\|^{2}_{h}=[\nabla_{h}U,\nabla_{h}U]_{\Omega}=[d_{x}U,d_{x}U]_{x}+[d_{y}U,d_{y}U]_{y},\\
&\dps\|U\|^{2}_{H^{1}_{h}}=\|U\|^{2}_{h}+\|\nabla_{h}U\|^{2}_{h}, &&  \|U\|_{\infty}=\max_{0\leq i\leq N}\sum_{j=0}^{N}|U_{i,j}|.
\end{align*}
From these above definitions, we obtain the following results.
\begin{lemma}[\!\cite{LSR19,WW88}]\label{intpat}
For any $U,V\in\mathcal{C}_{h}$, it holds
\begin{align}
-\big\langle \Delta_{h}U,V\big\rangle _{\Omega}=[\nabla_{h}U,\nabla_{h}V]_{\Omega}.
\end{align}
\end{lemma}
For the sake of the stability of the time-stepping schemes developed later, we introduce a stabilizing parameter $\kappa\geq 0$ and reform the TFAC model \eqref{TFAC} as follows:
\begin{align}\label{TFAC1}
_{0}{}\!D^{\alpha}_{t}\phi+\kappa\phi-\varepsilon^{2}\Delta\phi+f_{\kappa}(\phi)=0,
\end{align}
where $f_{\kappa}(\phi)=f(\phi)-\kappa\phi$. It is easy to see that the above system is equivalent to the original TFAC model \eqref{TFAC}.

\subsection{Linear stabilized L1 schemes}
 Let $\tau_n=t_n-t_{n-1}, n=1,\cdots,N$ be the time-step size,
 and $\tau=\max\{\Dt_{n},n=1,\cdots,N\}$ be the maximum step size.
We recalled the popular L1 approximation \cite{LX07,LYZ19,LLZ18} to discretize the Caputo fractional derivative at $t=t_{n+1}$, as follows
  \bry
 _{0}\!D^{\alpha}_{t_{n+1}}\phi
  =&\dps\frac{1}{\Gamma(1-\alpha)}\sum_{k=0}^{n}\int_{t_k}^{t_{k+1}}(t_{n+1}-s)^{-\alpha}\phi_{s}(s)ds\\[9pt]
  =&\dps\sum_{k=0}^{n}\frac{1}{\Gamma(1-\alpha)}\frac{\phi(t_{k+1})-\phi(t_{k})}{\Dt_{k+1}}\int_{t_{k}}^{t_{k+1}}(t_{n+1}-s)^{-\alpha}ds + R^{n+1}_{\tau}\\[12pt]
 :=& L^\alpha_{t} \phi(t_{n+1})+R^{n+1}_{\tau},
  \ery
where $R^{n+1}_{\tau}$ is the corresponding local truncation error and the $L^{\alpha}_{t}$ is defined by
\brr\label{eq1}
L^\alpha_{t} \phi(t_{n+1})
&\dps=\sum_{k=0}^{n}a^{(n)}_{n-k}(\phi(t_{k+1})-\phi(t_{k}))\\
&\dps=a^{(n)}_{0}\phi(t_{n+1})-\sum_{k=1}^{n}(a^{(n)}_{n-k}-a^{(n)}_{n-k+1})\phi(t_{k})-a^{n}_{n}\phi(t_{0})
\err
with discrete convolution kernels $a^{(n)}_{k}$ given by
\begin{align*}
&&  a^{(n)}_{n-k}=\frac{1}{\Gamma(1-\alpha)\Dt_{k+1}}\int_{t_{k}}^{t_{k+1}}(t_{n+1}-s)^{-\alpha}ds>0, k=0,1,\cdots,n.
\end{align*}
Moreover, the above $a^{(n)}_{n-k}$ are positive and decreasing, see also \cite{LLZ18,LYZ19}
\begin{align}
a^{(n)}_{n-k}>0,~~a^{(n)}_{n-k-1}>a^{(n)}_{n-k},~~k=0,1,\cdots,n-1.
\end{align}

Next, we recall two useful discrete tools in \cite{LTZ24,LZ21,YWCL22}: the discrete orthogonal convolution (DOC) kernels $\{\theta^{(n)}_{n-k}\}_{k=0}^{n}$  given by
\bq\label{eq_1}
\theta^{(n)}_{0}=\frac{1}{a^{(n)}_{0}} \mbox{ and } \theta^{(n)}_{n-k}=-\frac{1}{a^{(k)}_{0}}\sum_{j=k+1}^{n}\theta^{(n)}_{n-j}a^{(j)}_{j-k}\mbox{ for }0\leq k\leq n-1,
\eq
and the discrete complementary convolution (DCC) kernels  $\{p^{(n)}_{n-k}\geq0\}_{k=0}^{n}$:
\bq\label{eq_2}
p^{(n)}_{n-k}=\sum_{j=k}^{n}\theta^{(j)}_{j-k},~~ 0\leq k\leq n.
\eq
Then, it holds that
$$\dps\sum_{j=k}^{n}\theta^{(n)}_{n-j}a^{(j)}_{j-k}=\delta_{n,k},~~\sum_{j=k}^{n}p^{(n)}_{n-j}a^{(n)}_{j-k}=1\mbox{ for } 0\leq k\leq n.$$
Moreover, we have following useful lemma, which plays an important in derivation of the discrete variational energy law of the numerical schemes.
\begin{lemma}\cite{LTZ24,YWCL22}\label{lemm1}
For any real sequence $\{\psi^{k}\}_{k=0}^{n}$, we have
\beq
2\psi^{n}\sum_{k=0}^{n}a^{(n)}_{n-k}\psi^{k}\geq a^{(n)}_{0}|\psi^{n}|^{2}+ \sum_{k=0}^{n}p^{(n)}_{n-k}\Big(\sum_{j=0}^{k}a^{(k)}_{k-j}\psi^{j}\Big)^{2}-\sum_{k=0}^{n-1}p^{(n-1)}_{n-1-k}\Big(\sum_{j=0}^{k}a^{(k)}_{k-j}\psi^{j}\Big)^{2}.
\eeq
\end{lemma}

{\bf A first order scheme.}
Let $\Pi_{\mathcal{C}_h}$  be the operator that imposes a function onto $\mathcal{C}_h$ pointwise, and let $\vec{U}\in{\mathbb R}^{M^2}$ be the vector representation of $U\in\mathcal{C}_{h}$,
with the elements organized first along the $x$-direction and then along the $y$-direction.
We use
the equivalent system \eqref{TFAC1} to construct a first order linear fully-discrete stabilized L1  scheme for the TFAC equation \eqref{TFAC}:
given $\Phi^0=\Pi_{\mathcal{C}_h}\phi_0$, to find $\Phi^{n+1}\in\mathcal{C}_{h},$ such that
\begin{align}\label{scheme1}
 L^\alpha_t\Phi^{n+1}+\kappa\Phi^{n+1}-\varepsilon^{2}\Delta_{h}\Phi^{n+1}+f_{\kappa}(\Phi^{n})=0,\quad n=0,1,\cdots,N-1,
\end{align}
where $\Phi^{n}$ is a numerical approximation to $\phi(t_n)$ at the discrete grid points.
We simply denote the above linear stabilized first order L1 scheme as $\Phi^{n+ 1}= sFL1(\kappa,\Phi^n)$.
Furthermore, its vector form reads as follows:
  \bq\label{BDF1_tsor}
  L^\alpha_t\vPhi^{n+1}+\kappa\vPhi^{n+1}-\varepsilon^{2}D_{h}\vPhi^{n+1}+f_{\kappa}(\vPhi^{n})=0,
  \eq
where $f_{\kappa}(\vPhi^{n})=(\vPhi^{n}\big)^{.3}-(1+\kappa)\vPhi^{n}$
and the tensor form $D_{h}$ of the discrete $\Delta_{h}$ is given by
\bq\label{T_D}
D_{h}=I\otimes G+G\otimes I\in {\mathbb R}^{M^2\times M^2}
\eq
 with $I$ denoting the identity matrix (with the matched dimensions) and
\[ G=\frac{1}{h^{2}}
\begin{pmatrix}
-1&1&&&&\\
1&-2&1&&&\\
&\ddots &\ddots &\ddots&\\
&&1&-2&1&\\
&&&1&-1&\\
\end{pmatrix}_{M\times M}.\]
Let us first review several key lemmas that are fundamental for analyzing the discrete MBP of the proposed scheme \eqref{BDF1_tsor}.
Based on the definition of the variational energy $\mathcal{E}_{\alpha}[\phi]$ in \eqref{eqq1}, we introduce the corresponding discrete energy $\mathcal{E}_{\alpha,h}[\Phi^{n}]$ in the form
\begin{align}\label{dis_eg}
\mathcal{E}_{\alpha,h}[\phi^{n}]=E_{h}[\Phi^{n}]+\sum_{k=0}^{n-1}p^{(n-1)}_{n-1-k}\Big\|L^{\alpha}_{t}\Phi^{k}\Big\|^{2}_{h},
\end{align}
where the discrete energy $E_{h}[\Phi]$ is given by
\begin{align}\label{eqq2}
E_{h}[\Phi^{n}]\dps= & \frac{\varepsilon^{2}}{2}[\nabla_{h}\Phi^n,\nabla_{h}\Phi^n]_{\Omega}+\big\langle F(\Phi^{n}),1\big\rangle _{\Omega}\nonumber\\
\dps= & -\frac{h^{2}\varepsilon^{2}}{2}(\vPhi^{n})^{T}D_{h}\vPhi^{n}+h^{2}\sum_{i=1}^{M^{2}}F(\vPhi^{n}_{i}).
\end{align}
\begin{lemma}[\!\cite{HTY17,LTZ20,TY16}]\label{lemm2}
Assume that $B= (b_{i,j})$ is a real $M\times M$ matrix satisfying
\begin{align*}
 b_{i,i}<0,\quad |b_{i,i}|\geq \sum_{j=1,j\neq i}^{M}|b_{i,j}|,\quad  {i=1,2,\cdots,M.}
\end{align*}
Let  $A=aI-B$ where $a>0$ is a constant, then it holds that
\begin{align*}
\|A\overrightarrow{U}\|_{\infty}\geq a \|\overrightarrow{U}\|_{\infty},\quad \forall \overrightarrow{U}\in\R ^{M}.
\end{align*}
\end{lemma}
\begin{lemma}[\!\cite{HJQ22,TY16}]\label{para}
If the stabilizing parameter $\kappa$ satisfies
\bq \label{kappa_cd}
\kappa\geq \max_{\rho\in[-1,1]}f'(\rho),
\eq
 then it holds that
\bq\label{eqn1_4}
\big|f_{\kappa}(\rho)\big|\leq \kappa,\quad\forall \rho\in[-1,1].
\eq
\end{lemma}
Next, we establish the MBP preservation of the proposed first-order scheme \eqref{BDF1_tsor}, as stated in the following theorem.
\begin{theorem}\label{thmMBP}
Assume that $\|\vPhi^{0}\|_{\infty}\leq 1$ and the stabilizing parameter $\kappa$ satisfies \eqref{kappa_cd}.
Then, the first oder scheme \eqref{BDF1_tsor} conditionally preserves the discrete MBP, i.e.,
   $\|\vPhi^{n+1}\|_{\infty}\leq1$ for all $n=0,1,\cdots,N-1$. Moreover, it unconditionally satisfies the following discrete variational energy law:
   \bq\label{eqq4}
   \mathcal{E}_{\alpha,h}[\Phi^{n+1}]\leq \mathcal{E}_{\alpha,h}[\Phi^{n}],
   \eq
where $\mathcal{E}_{\alpha,h}[\Phi^{n}]$ is defined in \eqref{dis_eg}.
\end{theorem}
\begin{proof}
We prove the result by mathematical induction. Assume that the numerical solutions ${u^{k}{h}}$ satisfy the MBP for all $1\leq k\leq n$. We will then show that $|\vPhi^{,n+1}|{\infty}\leq 1$.
From \eqref{BDF1_tsor} and \eqref{eq1}, we have
 \beq
 \big((a^{(n)}_{0}+\kappa)I-\varepsilon^{2}D_{h}\big)\vPhi^{n+1}=\sum_{k=1}^{n}(a^{(n)}_{n-k}-a^{(n)}_{n-k+1})\vPhi^{k}+a^{(n)}_{n}\vPhi^{0}-f_{\kappa}(\vPhi^{n}).
 \eeq
Applying Lemma \ref{lemm2}, Lemma \ref{para}, and the triangle inequality yields
 \bry
 (a^{(n)}_{0}+\kappa)\big\|\vPhi^{n+1}\big\|_{\infty}&\leq\Big\|\big((a^{(n)}_{0}+\kappa)I-\varepsilon^{2}D_{h}\big)\vPhi^{n+1}\Big\|_{\infty}\\
 &\dps\leq \sum_{k=1}^{n}(a^{(n)}_{n-k}-a^{(n)}_{n-k+1})\big\|\vPhi^{k}\big\|_{\infty}+a^{(n)}_{n}\big\|\vPhi^{0}\big\|_{\infty}+\big\|f_{\kappa}(\vPhi^{n})\big\|_{\infty}\\
 &\leq a^{(n)}_{0}+\kappa,
 \ery
which implies the desired estimate $|\vPhi^{n+1}|_{\infty}\leq 1$.

Next, for any $U,V\in\mathcal{C}_{h}$, the differential mean value theorem together with \eqref{kappa_cd} gives
 \begin{align*}
 \big\langle F(U)-F(V),1\big\rangle _{\Omega}=&\big\langle f(V),U-V\big\rangle _{\Omega}+\big\langle \frac{1}{2}f'(\xi).*(U-V),U-V\big\rangle _{\Omega}\nonumber\\
 \leq&\big\langle f(V),U-V\big\rangle _{\Omega}+\frac{\kappa}{2}\big\|U-V\|^{2}_{h}\\
 =&\big\langle \kappa V+f_{\kappa}(V),U-V\big\rangle _{\Omega}+\frac{\kappa}{2}\big\|U-V\|^{2}_{h},\nonumber
 \end{align*}
where each entry of $\xi$ lies between the corresponding entries of $U$ and $V$.
 Moreover, using the identity $2a(a-b)=a^{2}-b^{2}+(a-b)^{2}$ and Lemma \ref{intpat}, we obtain
 \begin{align*}
 &\frac{\varepsilon^{2}}{2}\Big[[\nabla_{h}U,\nabla_{h}U]_{\Omega}-[\nabla_{h}V,\nabla_{h}V]_{\Omega}\Big]\nonumber\\
 &~~=\varepsilon^{2}[\nabla_{h}U,\nabla_{h}(U-V)]_{\Omega}-\frac{\varepsilon^{2}}{2}[\nabla_{h}(U-V),\nabla_{h}(U-V)]_{\Omega}\\
  &~~=\big\langle-\varepsilon^{2}\Delta_{h}U,U-V\big\rangle_{\Omega}-\frac{\varepsilon^{2}}{2}\|\nabla_{h}(U-V)\|^{2}_{h}.\nonumber
 \end{align*}
 Combining this with the definition of $E_{h}[\cdot]$ in \eqref{eqq2}, we obtain for any $U,V\in\mathcal{C}_{h}$,
\begin{align}\label{eqq3}
&E_{h}[U]-E_{h}[V]\nonumber\\
&~~\leq\big\langle(\kappa-\varepsilon^{2}\Delta_{h})U+f_{\kappa}(V),U-V\big\rangle_{\Omega}-\frac{\varepsilon^{2}}{2}\|\nabla_{h}(U-V)\|^{2}_{h}-\frac{\kappa}{2}\big\|U-V\|^{2}_{h},
\end{align}
Finally, invoking \eqref{eqq3}, \eqref{scheme1}, and Lemma \ref{lemm1}, we deduce
\begin{align*}
&E_{h}[\Phi^{n+1}]-E_{h}[\Phi^{n}]\nonumber\\
&\leq\big\langle(\kappa-\varepsilon^{2}\Delta_{h})\Phi^{n+1}+f_{\kappa}(\Phi^{n}),\Phi^{n+1}-\Phi^{n}\big\rangle_{\Omega}-\frac{\varepsilon^{2}}{2}\|\nabla_{h}(\Phi^{n+1}-\Phi^{n})\|^{2}_{h}\nonumber\\
&~~~-\frac{\kappa}{2}\big\|\Phi^{n+1}-\Phi^{n}\|^{2}_{h}\nonumber\\
&=-\big\langle L^{\alpha}_{t}\Phi^{n+1},\Phi^{n+1}-\Phi^{n}\big\rangle_{\Omega}-\frac{\varepsilon^{2}}{2}\|\nabla_{h}(\Phi^{n+1}-\Phi^{n})\|^{2}_{h}-\frac{\kappa}{2}\big\|\Phi^{n+1}-\Phi^{n}\|^{2}_{h}\\
&\leq-\sum_{k=0}^{n}p^{(n)}_{n-k}\big(L^{\alpha}_{t}\Phi^{k}\big)^{2}+\sum_{k=0}^{n-1}p^{(n-1)}_{n-1-k}\big(L^{\alpha}_{t}\Phi^{k}\big)^{2} -\frac{\varepsilon^{2}}{2}\|\nabla_{h}(\Phi^{n+1}-\Phi^{n})\|^{2}_{h}\nonumber\\
&~~~-(\frac{\kappa}{2}+a^{(n)}_{0})\big\|\Phi^{n+1}-\Phi^{n}\|^{2}_{h},\nonumber
\end{align*}
which establishes the desired inequality \eqref{eqq4}.
Then, we complete the proof.
\end{proof}

{\bf A linear high order predictor-corrector stabilized L1 scheme.}
In this subsection, we introduce a high-order fully discrete linear scheme for the TFAC equation \eqref{TFAC1}. Specifically, given $\Phi^0=\Pi_{\mathcal{C}_h}\phi_0$, and for $n=0,1\cdots,N-1$, we seek $\Phi^{n+1}\in\mathcal{C}_{h}$ such that
\begin{subequations}\label{L1_h}
    \begin{align}
        &\Phi^{*,n+1} = {\rm sFL1}(\kappa,\Phi^n), \label{L1_h1}\\
        &L^{\alpha}_{t}\Phi^{n+1}+\kappa\Phi^{n+1}-\varepsilon^{2}\Delta_{h}\Phi^{n+1}+f_{\kappa}(\Phi^{*,n+1})=0. \label{L1_h2}
    \end{align}
\end{subequations}
 The corresponding tensor form is given by
\begin{subequations}\label{L1_th}
    \begin{align}
        &\vPhi^{*,n+1} = {\rm sFL1}(\kappa,\vPhi^n), \label{L1_th1}\\
        &L^{\alpha}_{t}\vPhi^{n+1}+\kappa\vPhi^{n+1}-\varepsilon^{2}D_{h}\vPhi^{n+1}+f_{\kappa}(\vPhi^{*,n+1})=0. \label{L1_th2}
    \end{align}
\end{subequations}

The preservation of the MBP and the discrete variational energy dissipation law for the scheme \eqref{L1_h} is established in the following theorem.
\begin{theorem}\label{th1}
Under the assumption of Theorem \ref{thmMBP}, the proposed scheme \eqref{L1_h} is unconditionally MBP-preserving in the sense of $\|\vPhi^{n+1}\|_{\infty}\leq1$ for $n=0,1,\cdots,N-1.$ In addition, it holds that
\bq\label{eqq5}
   \mathcal{E}_{\alpha,h}[\Phi^{n+1}]\leq \mathcal{E}_{\alpha,h}[\Phi^{n}],\ \ n=0,1,\cdots,N-1
\eq
for any time stepping sizes, where $\mathcal{E}_{\alpha,h}[\Phi^{n}]$ is defined in \eqref{dis_eg}.
\end{theorem}
\begin{proof}Since the proof of the MBP preservation for scheme \eqref{L1_h} closely follows that of Theorem \ref{thmMBP}, we omit the details and leave them to the interested reader.
We now proceed to establish the discrete variational energy dissipation law \eqref{eqq5}.
From \eqref{eqq3} and \eqref{L1_h1}, it follows that
\begin{align}\label{eqq6}
&E_{h}[\Phi^{*,n+1}]-E_{h}[\Phi^{n}]\nonumber\\
\leq&\big\langle(\kappa-\varepsilon^{2}\Delta_{h})\Phi^{*,n+1}+f_{\kappa}(\Phi^{n}),\Phi^{*,n+1}-\Phi^{n}\big\rangle_{\Omega}-\frac{\varepsilon^{2}}{2}\|\nabla_{h}(\Phi^{*,n+1}-\Phi^{n})\|^{2}_{h}\nonumber\\
&-\frac{\kappa}{2}\big\|\Phi^{*,n+1}-\Phi^{n}\|^{2}_{h}\nonumber\\
=&-\big\langle a^{(n)}_{0}(\Phi^{*,n+1}-\Phi^{n})+\sum_{k=0}^{n-1}a^{(n)}_{n-k}(\Phi^{k+1}-\Phi^{k}),\Phi^{*,n+1}-\Phi^{n}\big\rangle_{\Omega}\\
&-\frac{\varepsilon^{2}}{2}\|\nabla_{h}(\Phi^{*,n+1}-\Phi^{n})\|^{2}_{h}-\frac{\kappa}{2}\big\|\Phi^{*,n+1}-\Phi^{n}\|^{2}_{h}\nonumber\\
:=&I_{1}-\frac{\varepsilon^{2}}{2}\|\nabla_{h}(\Phi^{*,n+1}-\Phi^{n})\|^{2}_{h}-\frac{\kappa}{2}\big\|\Phi^{*,n+1}-\Phi^{n}\|^{2}_{h}.\nonumber
\end{align}
Similarly, using \eqref{eqq3} and \eqref{L1_h2}, we obtain
\begin{align}\label{eqq7}
&E_{h}[\Phi^{n+1}]-E_{h}[\Phi^{*,n+1}]\nonumber\\
\leq&-\big\langle a^{(n)}_{0}(\Phi^{n+1}-\Phi^{n})+\sum_{k=0}^{n-1}a^{(n)}_{n-k}(\Phi^{k+1}-\Phi^{k}),\Phi^{n+1}-\Phi^{*,n+1}\big\rangle_{\Omega}\nonumber\\
&-\frac{\varepsilon^{2}}{2}\|\nabla_{h}(\Phi^{n+1}-\Phi^{*,n+1})\|^{2}_{h}-\frac{\kappa}{2}\big\|\Phi^{n+1}-\Phi^{*,n+1}\|^{2}_{h}\\
:=&I_{2}-\frac{\varepsilon^{2}}{2}\|\nabla_{h}(\Phi^{n+1}-\Phi^{*,n+1})\|^{2}_{h}-\frac{\kappa}{2}\big\|\Phi^{n+1}-\Phi^{*,n+1}\|^{2}_{h}.\nonumber
\end{align}
Moreover, we have
\begin{align}\label{eqq8}
I_{1}+I_{2}=&-\big\langle a^{(n)}_{0}(\Phi^{n+1}-\Phi^{n})+\sum_{k=0}^{n-1}a^{(n)}_{n-k}(\Phi^{k+1}-\Phi^{k}),\Phi^{n+1}-\Phi^{n}\big\rangle_{\Omega}\nonumber\\
&+\big\langle a^{(n)}_{0}(\Phi^{n+1}-\Phi^{*,n+1}),\Phi^{*,n+1}-\Phi^{n}\big\rangle_{\Omega}\nonumber\\
\leq&-\sum_{k=0}^{n}p^{(n)}_{n-k}\big(L^{\alpha}_{t}\Phi^{k}\big)^{2}+\sum_{k=0}^{n-1}p^{(n-1)}_{n-1-k}\big(L^{\alpha}_{t}\Phi^{k}\big)^{2}-\frac{a^{(n)}_{0}}{2}\|\Phi^{n+1}-\Phi^{n}\|^{2}_{h}\\
&-\frac{a^{n}_{0}}{2}\|\Phi^{n+1}-\Phi^{*,n+1}\|^{2}_{h} -\frac{a^{n}_{0}}{2}\|\Phi^{*,n+1}-\Phi^{n}\|^{2}_{h}+ \frac{a^{n}_{0}}{2}\|\Phi^{n+1}-\Phi^{n}\|^{2}_{h}\nonumber\\
\leq&-\sum_{k=0}^{n}p^{(n)}_{n-k}\big(L^{\alpha}_{t}\Phi^{k}\big)^{2}+\sum_{k=0}^{n-1}p^{(n-1)}_{n-1-k}\big(L^{\alpha}_{t}\Phi^{k}\big)^{2}-\frac{a^{n}_{0}}{2}\big[\|\Phi^{n+1}-\Phi^{*,n+1}\|^{2}_{h} \nonumber\\ &+\|\Phi^{*,n+1}-\Phi^{n}\|^{2}_{h}\big],\nonumber
\end{align}
where we have used Lemma \ref{lemm1} and the identity $2ab=-a^{2}-b^{2}+(a+b)^{2}$ for any $a,b\in \mathbb{R}$. Then, summing up \eqref{eqq6} and \eqref{eqq7}, together with \eqref{eqq8} and the definition of $\mathcal{E}_{\alpha,h}[\Phi^{n}]$ in \eqref{dis_eg}, we can obtain the desired discrete energy dissipation law \eqref{eqq5}. This ends the proof.
\end{proof}
\subsection{Fast linear predictor-corrector stabilized L1 scheme}

Due to the intrinsic non-locality of the standard L1 formula \eqref{eq1}, its direct application leads to significant computational and storage costs, rendering it impractical for long-time simulations.
To address this issue, we adopt the so-called sum-of-exponentials (SOE) technique \cite{Bey10} to approximate the history part of the time fractional derivative to speed up the calculation and reduce computational storage.
It has been shown in \cite{JZZ17,L10,LYZ19} that for any $0<\alpha<1$ and prescribed absolute tolerance $\epsilon\ll1$, there exist a positive integer $N_{\epsilon}$, positive quadrature nodes $\rho^{\alpha}_{i}$and weights $\omega^{\rho}_{i}$ $(1\leq i\leq N_{\epsilon})$ such that the kernel function $\frac{1}{\Gamma(1-\alpha)t^{\alpha}}$ admits an efficient SOE approximation of the form
  \beq
  \dps\Big|\frac{1}{\Gamma(1-\alpha)t^{\alpha}}-\sum_{i=1}^{N_{\epsilon}}\omega^{\alpha}_{i}e^{-\rho^{\alpha}_{i}t}\Big|\leq \epsilon, \qquad\forall t\in(\delta,T]
  \eeq
  with $0<\delta\ll 1.$
To develop the fast L1 formula using the SOE technique, we decompose the Caputo fractional derivative into two components: the local part and the history part.
The local part is discretized via linear interpolation, while the history part is approximated using the SOE technique. Specifically, we have
  \begin{align*}
 _{0}\!D^{\alpha}_{t_{n+1}}\phi
  =&\frac{1}{\Gamma(1-\alpha)}\int_{t_n}^{t_{n+1}}(t_{n+1}-s)^{-\alpha}\phi_{s}(s)+\frac{1}{\Gamma(1-\alpha)}\int_{0}^{t_{n}}(t_{n+1}-s)^{-\alpha}\phi_{s}(s)ds\\
\approx & a^{(n)}_{0}\big(\phi(t_{n+1})-\phi(t_{n})\big)+\int_{0}^{t_{n}} \sum_{i=1}^{N_{\epsilon}}\omega^{\alpha}_{i}e^{-\rho^{\alpha}_{i}(t_{n+1}-s)} \phi_{s}(s)ds\\
 =& a^{(n)}_{0}\big(\phi(t_{n+1})-\phi(t_{n})\big)+\sum_{i=1}^{N_{\epsilon}}\omega^{\alpha}_{i}e^{-\rho^{\alpha}_{i}\tau_{n+1}}\mathcal{H}_{i}(t_{n}),\ \ n=0,1,\cdots,
  \end{align*}
 where
\beq
\mathcal{H}_{i}(t_{n})=\int_{0}^{t_{n}}e^{-\rho^{\alpha}_{i}(t_{n}-s)}\phi_{s}(s)ds \mbox{ with } \mathcal{H}_{i}(t_{0})=0, \mbox{ for } l=0,1,\cdots, N_{\epsilon}.
\eeq
Using a recursive formula and linear interpolation in $(t_{n-1}, t_{n})$, $\mathcal{H}_{i}(t_{n})$ can be approximated by
\begin{align*}
\mathcal{H}_{i}(t_{n})&\dps=e^{-\rho^{\alpha}_{i}\Dt_{n}}\mathcal{H}_{i}(t_{n-1})+\int_{t_{n-1}}^{t_{n}}e^{-\rho^{\alpha}_{i}(t_{n}-s)}\phi_{s}(s)ds\\
&\dps\approx e^{-\rho^{\alpha}_{i}\Dt_{n}}\mathcal{H}_{i}(t_{n-1})+b_{i}^{n}(\phi(t_{n})-\phi(t_{n-1}))
\end{align*}
with
\bq\label{eqn1}
 b^{n}_{i}=\frac{1}{\Dt_{n}}\int_{t_{n-1}}^{t_{n}}e^{-\rho^{\alpha}_{i}(t_{n}-s)}ds,\mbox{ for } n=1,2,\cdots.
\eq
Then, we derive the discrete fast L1 formula:
\begin{align}\label{eqn2}
\mathcal{L}^{\alpha}_{t}\phi^{n+1}:=a^{(n)}_{0}\big(\phi^{n+1}-\phi^{n})\big)+\sum_{i=1}^{N_{\epsilon}}\omega^{\alpha}_{i}e^{-\rho^{\alpha}_{i}\tau_{n+1}}\mathcal{H}^{n}_{i},\ \ n=0,1, \cdots,
\end{align}
where $\mathcal{H}^{n}_{i}=\mathcal{H}^{n-1}_{i}+b_{i}^{n}(\phi^{n}-\phi^{n-1})$ with $\mathcal{H}^{n}_{i}=0$
for $i=1,2,\cdots,N_{\epsilon}$. Consequently, it follows from \eqref{eqn1} and \eqref{eqn2} that
\begin{align}
\mathcal{L}^{\alpha}_{t}\phi^{n+1}&=a^{(n)}_{0}\big(\phi^{n+1}-\phi^{n}\big)+\sum_{i=1}^{N_{\epsilon}}\omega^{\alpha}_{i}\sum_{k=1}^{n}e^{-\rho^{\alpha}_{i}(t_{n+1}-t_{k})}b^{k}_{i}(\phi^{k}-\phi^{k-1})\nonumber\\
&=a^{(n)}_{0}\big(\phi^{n+1}-\phi^{n}\big)+\sum_{k=1}^{n}\frac{1}{\Dt_{k}}\sum_{i=1}^{N_{\epsilon}}\omega^{\alpha}_{i}\int_{t_{k-1}}^{t_{k}}e^{-\rho^{\alpha}_{i}(t_{n+1}-s)}ds (\phi^{k}-\phi^{k-1})\\
&:=\sum_{k=0}^{n}A^{(n)}_{n-k}(\phi^{k+1}-\phi^{k}),\nonumber
\end{align}
where the discrete convolution kernel $A^{(n)}_{n-k}$ is defined by
\bq\label{kernel}
A^{(n)}_{0}=a^{(n)}_{0}, A^{(n)}_{n-k}=\frac{1}{\Dt_{k+1}}\sum_{i=1}^{N_{\epsilon}}\omega^{\alpha}_{i}\int_{t_{k}}^{t_{k+1}}e^{-\rho^{\alpha}_{i}(t_{n+1}-s)}ds \mbox{ for } k=0,1,\cdots,n-1.
\eq
Then, fast linear predictor-corrector stabilized L1 scheme for the TFAC problem \eqref{TFAC1} reads: given $\Phi^0=\Pi_{\mathcal{C}_h}\phi_0$, and for $n=0,1\cdots,N-1$, find $\Phi^{n+1}\in\mathcal{C}_{h}$ such that
\begin{subequations}\label{fL1_h}
    \begin{align}
        & \mathcal{L}^\alpha_t\Phi^{*,n+1}+\kappa\Phi^{*,n+1}-\varepsilon^{2}\Delta_{h}\Phi^{*,n+1}+f_{\kappa}(\Phi^{n})=0, \label{fL1_h1}\\
        &\mathcal{L}^{\alpha}_{t}\Phi^{n+1}+\kappa\Phi^{n+1}-\varepsilon^{2}\Delta_{h}\Phi^{n+1}+f_{\kappa}(\Phi^{*,n+1})=0, \label{fL1_h2}
    \end{align}
\end{subequations}
 and its tensor form is presented as follows
\begin{subequations}\label{fL1_th}
    \begin{align}
        & \mathcal{L}^\alpha_t\vPhi^{*,n+1}+\kappa\vPhi^{*,n+1}-\varepsilon^{2}D_{h}\vPhi^{*,n+1}+f_{\kappa}(\vPhi^{n})=0, \label{fL1_th1}\\
        &\mathcal{L}^{\alpha}_{t}\vPhi^{n+1}+\kappa\vPhi^{n+1}-\varepsilon^{2}D_{h}\vPhi^{n+1}+f_{\kappa}(\vPhi^{*,n+1})=0, \label{fL1_th2}
    \end{align}
\end{subequations}
where $D_{h}$ is defined in \eqref{T_D}.

It is shown in \cite[Lemma 2.5]{LYZ19} that the discrete convolution kernels $\{A^{n}_{k}\}_{k=0}^{n}$ satisfy following property, which plays a crucial role in deriving both the MBP preservation and the discrete variational energy dissipation for the fast linear predictor-corrector stabilized L1 scheme \eqref{fL1_h}.
\begin{lemma}\label{lem3}
If the tolerance error $\epsilon$ of SOE satisfies $\epsilon\leq\min\{ \frac{T^{-\alpha}}{3\Gamma(1-\alpha)}, \frac{\alpha}{\Gamma(2-\alpha)}\}$, then
the discrete convolutional kernels $\{A^{n}_{k}\}_{k=0}^{n}$  defined in \eqref{kernel}, satisfy

$(i)$ $A^{(n)}_{k-1}>A^{(n)}_{k}>0$ for $0\leq k\leq n$;

$(ii)$  $A^{(n)}_{0}=a^{(n)}_{0}$ and $A^{(n)}_{n-k}\geq \frac{2}{3}a^{(n)}_{n-k}$ for $0\leq k\leq n-1.$
\end{lemma}
Similarly to \eqref{eq_1} and \eqref{eq_2}, the DOC kernels $\{\hat{\theta}^{(n)}_{n-k}\}_{k=0}^{n-1}$ and the DCC kernels  $\{\hat{p}^{(n)}_{n-k}\geq0\}_{k=0}^{n-1}$ associated with the discrete convolutional kernels $\{A^{n}_{k}\}_{k=0}^{n}$  are given as follows:
\begin{align*}
\hat{\theta}^{(n)}_{0}&=\frac{1}{A^{(n)}_{0}} \mbox{ and } \hat{\theta}^{(n)}_{n-k}=-\frac{1}{A^{(k)}_{0}}\sum_{j=k+1}^{n}\hat{\theta}^{(n)}_{n-j}A^{(j)}_{j-k}\mbox{ for }0\leq k\leq n-1,\nonumber\\
\hat{p}^{(n)}_{n-l}&=\sum_{j=l}^{n}\hat{\theta}^{(j)}_{j-l},~~ 0\leq l\leq n.
\end{align*}
Then, it holds that
\beq
\dps\sum_{j=k}^{n}\hat{\theta}^{(n)}_{n-j}A^{(j)}_{j-k}=\delta_{n,k},~~\sum_{j=k}^{n}\hat{p}^{(n)}_{n-j}A^{(n)}_{j-k}=1,~~ \hat{p}^{(n)}_{n-k}\geq0 \mbox{ for }, 0\leq k\leq n.
\eeq
Furthermore, for any real sequence $\{\psi^{k}\}_{k=0}^{n}$, we have
\beq
2\psi^{n}\sum_{k=0}^{n}A^{(n)}_{n-k}\psi^{k}\geq A^{(n)}_{0}|\psi^{n}|^{2}+ \sum_{k=0}^{n}\hat{p}^{(n)}_{n-k}\Big(\sum_{j=0}^{k}A^{(k)}_{k-j}\psi^{j}\Big)^{2}-\sum_{k=0}^{n-1}\hat{p}^{(n-1)}_{n-1-k}\Big(\sum_{j=0}^{k}A^{(k)}_{k-j}\psi^{j}\Big)^{2},
\eeq
seeing also Lemma 4.1 in \cite{YWCL22}.
Following a similar proof to that of Theorem \ref{th1}, we can also establish the preservation of the discrete MBP and the discrete variational energy dissipation law for the proposed scheme \eqref{fL1_h}, as stated in the following theorem.
\begin{theorem}\label{them1}
Under the assumption of Theorem \ref{thmMBP} and Lemma \ref{lem3}, the proposed scheme \eqref{fL1_h} is unconditional MBP-preserving, i.e., $\|\vPhi^{n+1}\|_{\infty}\leq1$ for $n=0,1,\cdots,N-1.$ Moreover, it holds that
\bq
   \mathcal{E}_{\alpha,h}[\Phi^{n+1}]\leq \mathcal{E}_{\alpha,h}[\Phi^{n}],\ \ n=0,1,\cdots,N-1
\eq
for any time stepping sizes, where $\mathcal{E}_{\alpha,h}[\Phi^{n}]$ is defined by
\begin{align}\label{dis_eg_f}
\mathcal{E}_{\alpha,h}[\phi]=E_{h}[\Phi^{n}]+\sum_{k=0}^{n}\hat{p}^{(n-1)}_{n-1-k}\Big\|L^{\alpha}_{t}\Phi^{k}\Big\|^{2}_{h}
\end{align}
with $E_{h}[\Phi^{n}]$ given in \eqref{eqq2}.
\end{theorem}
\section{Error analysis}
\setcounter{equation}{0}
In this section, we investigate the error estimates for the proposed schemes. For simplicity, we focus on the convergence analysis of the fast linear predictor-corrector stabilized L1 scheme \eqref{fL1_h}. The corresponding results can be easily extended to the case of the linear predictor-corrector stabilized L1 scheme \eqref{L1_h}.  Before deriving the error estimate of the scheme  \eqref{fL1_h}, we review some useful lemmas reported in \cite{LMZ19}  to extend their results on fractional discrete  Gr\"onwall inequality in Theorem 3.1 with $\theta=1$.
This extension plays a crucial role in removing the time-step size restriction in the error analysis of the proposed linear schemes, a limitation inherent in using the standard fractional discrete Gr\"onwall inequality with $0 \leq \theta < 1$ in \cite{LMZ19}.
\begin{lemma}\label{lem1}
Let the following assumptions $A1$ and $A2$ hold:

$A1.$ The discrete kernels $\{B^{(n)}_{k}\}_{k=0}^{n}$ are positive and monotone dicreasing, i.e.,
\begin{align*}
B^{(n)}_{0}\geq B^{(n)}_{1}\geq\cdots\geq B^{(n)}_{n}>0, \mbox{ for } 0\leq n\leq N-1.
\end{align*}

$A2.$ There exists a positive constant $\pi_{B}$ such that
\begin{align*}
B^{(n)}_{n-k}\geq\frac{1}{\pi_{B}\tau_{k+1}}\int_{t_{k}}^{k+1}\omega_{1-\alpha}(t_{n+1}-s)ds \mbox{ for } 0\leq k\leq n.
\end{align*}
Then, it holds that
\begin{align}
0\leq \widetilde{p}^{n}_{n-k} \leq \pi_{B}\gamma(2-\alpha)\tau_{n+1} \mbox{ for } 0\leq n\leq N-1,
\end{align}
where $\{\widetilde{p}^{n}_{n-k}\}_{k=0}^{n}$ is the DCC kernels of the sequence $\{B^{(n)}_{k}\}_{k=0}^{n}$.
\end{lemma}
\begin{lemma}\label{lem2}
Assume the $A1$, $A2$ and the following assumption A3 hold

$A3.$ There is a constant $\rho$ such that the step size ratios $\rho_{k}=\tau_{k}/\tau_{k+1}$ satisfy
\begin{align*}
\rho_{k}\leq \rho, \mbox{ for } 1\leq k\leq N-1.
\end{align*}
Then, for any real $\mu>0$, it holds that
\begin{align*}
\sum_{k=0}^{n-1}\widetilde{p}^{n}_{n-k}E_{\alpha}(\mu t^{\alpha}_{k+1})\leq\pi_{B}\max(1,\rho)\frac{E_{\alpha}(\mu t^{\alpha}_{n+1})-1}{\mu},
\end{align*}
where $E_{\alpha}(t)$ is the Mittag--Leffler function defined by
\begin{align*}
E_{\alpha}(t):=\sum_{k=0}^{\infty}\frac{t^{k}}{\Gamma(1+k\alpha)}.
\end{align*}
\end{lemma}

\begin{theorem}\label{thm1}
Let the assumptions $A1-A3$ hold. Assume further that $\{g^{n}\}_{n=1}^{N}$ and $\{\lambda_{l}\}_{l=0}^{N-1}$ are two nonnegative sequences, and there is a positive constant $\Lambda$ such that $\sum_{l=0}^{N-1}\lambda_{l}\leq \Lambda.$
Then, for any nonnegative sequence $\{\psi_{k}\}_{k=0}^{N}$ such that
\begin{align}\label{equ1}
\sum_{k=0}^{n}B^{(n)}_{n-k}(\psi^{k+1}-\psi^{k})\leq \sum_{k=0}^{n}\lambda_{n-k} \psi^{k}+g^{n+1} \mbox{ for } 0\leq n\leq N-1,
\end{align}
it holds that
\begin{align}\label{equ4}
\psi^{n+1}\leq E_{\alpha}(\max(1,\rho)\pi_{B}\Lambda t^{\alpha}_{n+1})(\psi^{0}+\max_{0\leq k\leq n}\sum_{j=0}^{k}\widetilde{p}^{(k)}_{k-j}g^{j+1}).
\end{align}
\end{theorem}
\begin{proof}
B replacing the index $n$ with $j$ in \eqref{equ1}, and multiplying both sides by $\widetilde{p}^{n}_{n-j}$ and summing up $j$ from 0 to $n$, we get
\begin{align}
\sum_{j=0}^{n}\widetilde{p}^{n}_{n-j}\sum_{k=0}^{j}B^{(j)}_{j-k}(\psi^{k+1}-\psi^{k})\leq\sum_{j=0}^{n}\widetilde{p}^{n}_{n-j}\sum_{k=0}^{j}\lambda_{j-k}\psi^{k}+\sum_{j=0}^{n}\widetilde{p}^{n}_{n-j}g^{n+1}
\end{align}
for $0\leq n\leq N-1.$ For righthand side term of the above inequality, we exchange the order of summation and use the identity $\sum_{j=k}^{n}\widetilde{p}^{n}_{n-j}B^{(j)}_{j-k}\equiv1$ to obtain that
\begin{align}\label{equ2}
\sum_{j=0}^{n}\widetilde{p}^{n}_{n-j}\sum_{k=0}^{j}B^{(j)}_{j-k}(\psi^{k+1}-\psi^{k})&=\sum_{k=0}^{n}(\psi^{k+1}-\psi^{k})\sum_{j=k}^{n}\widetilde{p}^{n}_{n-j}B^{(j)}_{j-k}\nonumber\\
&=\sum_{k=0}^{n}(\psi^{k+1}-\psi^{k})\\
&=\psi^{n+1}-\psi^{0}.\nonumber
\end{align}
From \eqref{equ1} and \eqref{equ2}, it follows that
\begin{align}\label{equ3}
\psi^{n+1}\leq \sum_{j=0}^{n}\widetilde{p}^{n}_{n-j}\sum_{k=0}^{j}\lambda_{j-k}\psi^{k}+\psi^{0}+\sum_{j=0}^{n}\widetilde{p}^{n}_{n-j}g^{n+1}.
\end{align}
For brevity, we simply denote the result \eqref{equ4} as
\begin{align}\label{equ5}
\psi^{k+1}\leq F_{k+1}G_{k+1},\qquad 0\leq k \leq N-1,
\end{align}
where the $F_{k}$ and $G_k$ are given by
\beq
F_{k+1}:=E_{\alpha}(\max(1,\rho)\pi_{B}\Lambda t^{\alpha}_{k+1}),\quad G_{k+1}:=\psi^{0}+\max_{0\leq l\leq k}\sum_{j=0}^{l}\widetilde{p}^{(l)}_{l-j}g^{j+1}.
\eeq
From the definition of $F_{k+1}$ and $G_{k+1}$, it can be verified that both $F_{k+1}$ and $G_{k+1}$ are nonnegative and increasing with respect to the index $k$.

Next, we use mathematical induction approach to derive the desired result \eqref{equ5}. For k=0, it follows from \eqref{equ3} that
\begin{align*}
\psi^{1}&\leq\widetilde{p}^{0}_{0}\lambda_{0}\psi^{0}+(\psi^{0}+\widetilde{p}^{0}_{0}g^{1})\nonumber\\
&\leq \pi_{B}\Gamma(2-\alpha)\tau^{\alpha}_{1}\Lambda\psi^{0}+(\psi^{0}+\widetilde{p}^{0}_{0}g^{1})\nonumber\\
&\leq (1+\pi_{B}\Gamma(2-\alpha)\tau^{\alpha}_{1}\Lambda)(\psi^{0}+\widetilde{p}^{0}_{0}g^{1})\\
&\leq(1+\max(1,\rho)\pi_{B}\tau^{\alpha}_{1}\Lambda)(\psi^{0}+\widetilde{p}^{0}_{0}g^{1})\nonumber\\
&\leq E_{\alpha}(\max(1,\rho)\pi_{B}\tau^{\alpha}_{1}\Lambda)(\psi^{0}+\widetilde{p}^{0}_{0}g^{1})\nonumber\\
&\leq F_{1}G_{1},\nonumber
\end{align*}
where we have used the Lemma \ref{lem1} and the identity $1+x\leq E_{\alpha}(x)$ for any $x\geq0.$ Then, we assume \eqref{equ5} holds for $0\leq k\leq n$ to derive from \eqref{equ3} that
\begin{equation}\label{equ6}
\begin{aligned}
\psi^{n+1}&\leq \sum_{j=0}^{n}\widetilde{p}^{n}_{n-j}\sum_{k=0}^{j}\lambda_{j-k}F_{k}G_{k}+\psi^{0}+\sum_{j=0}^{n}\widetilde{p}^{n}_{n-j}g^{n+1}\\
&\leq \sum_{j=0}^{n}\widetilde{p}^{n}_{n-j}F_{j}G_{j}\sum_{k=0}^{j}\lambda_{j-k}+G_{n+1}\\
&\leq G_{n+1}\Lambda\sum_{j=0}^{n}\widetilde{p}^{n}_{n-j}F_{j}+G_{n+1}\\
&= G_{n+1}\Lambda\sum_{j=0}^{n}\widetilde{p}^{n}_{n-j}E_{\alpha}(\max(1,\rho)\pi_{B}\Lambda\tau^{\alpha}_{j})+G_{n+1}\\
&\leq G_{n+1}\Lambda\pi_{B}\max(1,\rho)\frac{E_{\alpha}(\max(1,\rho)\pi_{B}\Lambda\tau^{\alpha}_{n+1})-1}{\max(1,\rho)\pi_{B}\Lambda}+G_{n+1}\\
&= F_{n+1}G_{n+1},
\end{aligned}
\end{equation}
where we have used Lemma \ref{lem2}.
Then, we complete the proof.
\end{proof}

Let $\Phi(t)= \Pi_{\mathcal{C}_h}\phi(t)$ where $ \phi$
denotes the exact solution of \eqref{TFAC1}. We also use  $C$, $C_{\phi}$ and $C_i$ to denote some needed generic positive constants independent of $h$ and $\Dt$. The local truncation errors $R^{n+1}_{\alpha}$ and $R^{n+1}_{f_{\kappa}}$ for $0\leq n\leq N-1$ are given by
\begin{equation}\label{truerrs}
 \begin{aligned}
 \overrightarrow{R}^{n+1}_{\alpha}&=\!\!_{0}D^{\alpha}_{t_{n+1}}\vPhi-\sum_{k=0}^{n}A^{(n)}_{n-k}(\vPhi(t_{k+1})-\vPhi(t_{k})),\\
  \overrightarrow{R}_{f_\kappa}^{n+1}&=f_{\kappa}(\vPhi(t_{n+1}))-f_{\kappa}(\vPhi(t_{n})).
\end{aligned}
\end{equation}
With a reasonable requirement on the exact solution $\phi$ of the problem \eqref{TFAC}, we have following estimate for the above local truncation error, seeing also Lemma 3.3 and Lemma 3.4 in \cite{LYZ19}.
\begin{lemma}\label{lem4}
Under the assumption of Lemma \ref{lem3}, we assume the continuous solution $\phi$ of the problem \eqref{TFAC} satisfies
\beq
\|\phi'(t)\|_{W^{0,\infty}(\Omega)}\leq C_{\phi}(1+t^{\sigma-1}),\quad \|\phi''(t)\|_{W^{0,\infty}(\Omega)}\leq C_{\phi}(1+t^{\sigma-2}),
\eeq
where $\sigma\in(0,1)$ is a real parameter. Then, it holds for $0\leq n\leq N-1$ that
\begin{align*}
\sum_{j=0}^{n}\widehat{p}^{n}_{n-j}|R^{j+1}_{\alpha}|\leq C_{\phi}(\frac{\tau^{\sigma}_{1}}{\sigma}+\frac{1}{1-\alpha}\max_{1\leq k\leq n}t^{\alpha}_{k+1}t^{\sigma-2}_{k}\tau^{2-\alpha}_{k+1}+\frac{\epsilon}{\sigma}t^{\alpha}_{n+1}\widehat{t}^{2}_{n+1}),
\end{align*}
where $\{\widehat{p}^{n}_{n-j}\}_{j=0}^{n}$ is the DCC kernels  of the discrete convolutional kernels $\{A^{n}_{k}\}_{k=0}^{n}$ and $\hat{t}_{n+1}=\max(1,t_{n+1})$. Moreover, if the temporal mesh satisfies the following assumption {\bf{AssG}} in \cite{LYZ19}:

AssG. Let $\gamma\geq 1$ be a user--chosen parameter. There is a mesh--independent constant $C_{\gamma}>0$
such that $\tau_{k}\leq C_{\gamma}\tau\min(1,t_{k}^{1-1/\gamma})$ for $1\leq k\leq N$ and $t_{k}\leq C_{\gamma}t_{k-1}$ for $2\leq k\leq N.$ Then it holds that
\begin{align*}
\sum_{j=0}^{n}\widehat{p}^{n}_{n-j}|R^{j+1}_{\alpha}|\leq C_{\phi}(\tau^{\min(2-\alpha,\gamma\sigma)}+\frac{\epsilon}{\sigma}t^{\alpha}_{n+1}\widehat{t}^{2}_{n+1}), \quad 0\leq n\leq N-1,
\end{align*}
\end{lemma}
\begin{lemma}\label{lemm5}
Under the assumption of Lemma \ref{lem4}, we assume the nonlinear term $f_{\kappa}(\cdot)\in C^{1}(\mathbb{R})$, then it holds for $0\leq n\leq N-1$ that
\begin{align*}
\sum_{j=0}^{n}\widehat{p}^{n}_{n-j}\frac{|R^{j+1}_{f_{\kappa}}|}{A^{(j)}_{0}}\leq C_{\phi}\tau^{2\alpha}_{1}(\tau_{1}+\frac{\tau^{\sigma}_{1}}{\sigma})+C_{\phi}t_{n+1}^{\alpha}\max_{1\leq k\leq n}(\tau^{1
+\alpha}_{k+1}+t_{k}^{\sigma-1}\tau^{1+\alpha}_{k+1}).
\end{align*}
In addition, if the temporal mesh satisfies the AssG, we have
\begin{align*}
\sum_{j=0}^{n}\widehat{p}^{n}_{n-j}\frac{|R^{j+1}_{f_{\kappa}}|}{A^{(j)}_{0}}\leq C_{\phi}\tau^{\min(1+\alpha,\gamma\sigma)},\quad 0\leq n\leq N-1.
\end{align*}
\end{lemma}
\begin{proof}
Since the proof can be easily obtained by following a process similar to that used in the proof of Lemma 3.4 in \cite{LYZ19}, we omit the details here and leave them to interested readers.
\end{proof}

Now we are ready to derive error estimate for the proposed fast linear predictor-corrector stabilized L1 scheme \eqref{fL1_h}. Define the errors  $\vec{e}^{n}=\vPhi^{n}-\vPhi(t_{n})$ and $\vec{e}^{*,n}=\vPhi^{*,n}-\vPhi(t_{n})$.
We have following convergence estimate for the scheme \eqref{fL1_h}.
\begin{theorem}\label{Thm1}
Under the assumption of Theorem \ref{them1} and Lemma \ref{lemm5}, and the A3 holds, then it holds for the proposed scheme \eqref{fL1_h} that
\begin{equation}\label{ex1}
\begin{aligned}
\|\vec{e}^{n+1}\|_{\infty}\leq& C_{\phi}\Big(\tau^{1+2\alpha}_{1}+\frac{\tau^{\sigma}_{1}}{\sigma}+t_{n+1}^{\alpha}\big(\max_{1\leq k\leq n}(\tau^{1
+\alpha}_{k+1}+t_{k}^{\sigma-1}\tau^{1+\alpha}_{k+1})\\
&+\max_{1\leq k\leq n}t^{\sigma-2}_{k}\tau^{2-\alpha}_{k+1}\big)+\frac{\epsilon}{\sigma}t^{\alpha}_{n+1}\widehat{t}^{2}_{n+1}+h^{2}\Big), \quad 0\leq n\leq N-1.
\end{aligned}
\end{equation}
Moreover, when the temporal mesh satisfies AssG, we have
\begin{align}\label{ex2}
\|\vec{e}^{n+1}\|_{\infty}\leq C_{\phi}\big(\tau^{\min\{1+\alpha,2-\alpha,\gamma\sigma\}}+\epsilon+h^{2}\big), \quad 0\leq n\leq N-1.
\end{align}
\end{theorem}
\begin{proof}
The exact solution $\vPhi(\cdot)$ satisfies
\brr\label{equa2}
&\dps \sum_{j=0}^{n} A^{(n)}_{n-j}(\vPhi(t_{j+1})-\vPhi(t_{j}))-\varepsilon^{2}D_{h}\vPhi(t_{n+1})+\kappa\vPhi(t_{n+1})+f_{\kappa}(\vPhi(t_{n+1}))\\
&\quad\dps+\overrightarrow{R}^{n+1}_{\alpha}+\overrightarrow{R}^{n+1}_{\Delta}=0
\err
for any $1\leq n\leq N-1$, where  $\overrightarrow{R}^{n+1}_{\alpha}$ is defined in \eqref{truerrs} and $\overrightarrow{R}^{n+1}_{\Delta}$  is expressed as
\beq
 \overrightarrow{R}^{n+1}_{\Delta}\dps=-\varepsilon^2\Delta \vPhi(t_{n+1})+\varepsilon^2 D_{h}\vPhi(t_{n+1}).
\eeq
Moreover, it is straightforward to verify that
\begin{align}\label{eqnn5}
\|\overrightarrow{R}^{n+1}_{\Delta}\|_{\infty}\dps\leq  \varepsilon^{2}\frac{h^{2}}{6}\|\phi\|^{2}_{L^{\infty}(0,T;W^{4,\infty}(\Omega))}.
\end{align}
Combining \eqref{equa2} with \eqref{fL1_th2}, the error equation of $\ve^{n+1}$ reads as
\begin{equation}\label{equa5}
\begin{aligned}
&\dps A^{(n)}_{0}\ve^{n+1}-\sum_{k=1}^{n}(A^{n}_{n-k}-A^{(n)}_{n-k+1})\ve^{k}-A^{(n)}_{n}\ve^{0}+\kappa\ve^{n+1}-\varepsilon^{2}D_{h}\ve^{n+1} \\
&+f_{\kappa}(\vPhi^{*,n+1})-f_{\kappa}(\vPhi(t_{n+1}))= \overrightarrow{R}^{n+1}_{\alpha}+\overrightarrow{R}^{n+1}_{\Delta}.
\end{aligned}
\end{equation}
Then, the above equality can be rewritten as:
\begin{align*}
&\dps A^{(n)}_{0}\ve^{n+1}+\kappa\ve^{n+1}-\varepsilon^{2}D_{h}\ve^{n+1} \\
= &\sum_{k=0}^{n}(A^{n}_{n-k}-A^{(n)}_{n-k+1})\ve^{k}+f_{\kappa}(\vPhi(t_{n+1}))-f_{\kappa}(\vPhi^{*,n+1})+\overrightarrow{R}^{n+1}_{\alpha}+\overrightarrow{R}^{n+1}_{\Delta}.
\end{align*}
where we set $A^{(n)}_{n+1}=0$.
Furthermore, we deduce from Lemma \ref{lemm2} amd Lemma \ref{lem3} that
\begin{align*}
&A^{(n)}_{0}\|\ve^{n+1}\|_{\infty}\\
\leq&  \Big\|A^{(n)}_{0}\ve^{n+1}+\kappa\ve^{n+1}-\varepsilon^{2}D_{h}\ve^{n+1} \Big\|_{\infty}\\
=&  \|\sum_{k=0}^{n}(A^{n}_{n-k}-A^{(n)}_{n-k+1})\ve^{k}+f_{\kappa}(\vPhi(t_{n+1}))-f_{\kappa}(\vPhi^{*,n+1})+\overrightarrow{R}^{n+1}_{\alpha}+\overrightarrow{R}^{n+1}_{\Delta}\|_{\infty}\\
\leq& \sum_{k=0}^{n}(A^{n}_{n-k}-A^{(n)}_{n-k+1})\|\ve^{k}\|_{\infty}+\max_{\rho\in[-1,1]}|f'_{\kappa}(\rho)|\|\ve^{*,n+1})\|_{\infty}+\|\overrightarrow{R}^{n+1}_{\alpha}\|_{\infty}+\|\overrightarrow{R}^{n+1}_{\Delta}\|_{\infty}.
\end{align*}
From the above inequality, it follows that
\begin{align}\label{equa4}
&\sum_{k=0}^{n}A^{(n)}_{n-k}(\|\ve^{k+1}\|_{\infty}-\|\ve^{k}\|_{\infty})\nonumber\\
&\leq \max_{\rho\in[-1,1]}|f'_{\kappa}(\rho)|\|\ve^{*,n+1})\|_{\infty}+\|\overrightarrow{R}^{n+1}_{\alpha}\|_{\infty}+\|\overrightarrow{R}^{n+1}_{\Delta}\|_{\infty}.
\end{align}
Following the similar process of deriving \eqref{equa5}, we can easily obtain the error equation of $ \ve^{*,n+1}$ from \eqref{equa2} and \eqref{fL1_th1}:
\begin{equation}\label{eqn1_8}
\begin{aligned}
&\dps A^{(n)}_{0}\ve^{*,n+1}-\sum_{k=1}^{n}(A^{n}_{n-k}-A^{(n)}_{n-k+1})\ve^{k}-A^{(n)}_{n}\ve^{0}+\kappa\ve^{n+1}-\varepsilon^{2}D_{h}\ve^{*,n+1} \nonumber\\
&+f_{\kappa}(\vPhi^{n})-f_{\kappa}(\vPhi(t_{n+1}))
= \overrightarrow{R}^{n+1}_{\alpha}+\overrightarrow{R}^{n+1}_{\Delta}.
\end{aligned}
\end{equation}
Multiplying \eqref{eqn1_8} with $1/A^{(n)}_{0}$, and using Lemma \ref{lemm2} amd Lemma \ref{lem3}, gives
 \begin{align*}
&\| \ve^{*,n+1}\|_{\infty}\leq \dps\big\| \ve^{*,n+1}+\frac{\kappa}{A^{(n)}_{0}}\ve^{*,n+1}-\frac{\varepsilon^{2}}{A^{(n)}_{0}}\Dt_{n+1}D_{h} \ve^{*,n+1}\big\|_{\infty}\\
 &=\frac{1}{A^{(n)}_{0}}  \Big\|\sum_{k=0}^{n}(A^{(n)}_{n-k}-A^{(n)}_{n-k+1})\ve^{k}+f_{\kappa}(\vPhi(t_{n+1}))-f_{\kappa}(\vPhi^{n})+\overrightarrow{R}^{n+1}_{\alpha}+\overrightarrow{R}^{n+1}_{\Delta}\Big\|_{\infty}\\
&\leq\frac{1}{A^{(n)}_{0}}\Big[\sum_{k=0}^{n}(A^{(n)}_{n-k}-A^{(n)}_{n-k+1})\|\ve^{k}\|_{\infty}+\|\overrightarrow{R}^{n+1}_{f_{\kappa}}\|_{\infty}+\max_{\rho\in[-1,1]}|f'_{\kappa}(\rho)|\|\ve^{n}\|_{\infty}\\
&~~+\|\overrightarrow{R}^{n+1}_{\alpha}\|_{\infty}+\|\overrightarrow{R}^{n+1}_{\Delta}\|_{\infty}\Big].
\end{align*}
Then, we can deduce from the above error inequality and \eqref{equa4} that
\begin{align*}
&\sum_{k=0}^{n}A^{(n)}_{n-k}(\|\ve^{k+1}\|_{\infty}-\|\ve^{k}\|_{\infty})\\
\leq& \frac{\max_{\rho\in[-1,1]}|f'_{\kappa}(\rho)|}{A^{(n)}_{0}}\Big[\sum_{k=0}^{n}(A^{(n)}_{n-k}-A^{(n)}_{n-k+1})\|\ve^{k}\|_{\infty}+\max_{\rho\in[-1,1]}|f'_{\kappa}(\rho)|\|\ve^{n}\|_{\infty}+\overrightarrow{R}^{n+1}_{f_{\kappa}}\Big]\\
&+\Big(1+\frac{\max_{\rho\in[-1,1]}|f'_{\kappa}(\rho)|}{A^{(n)}_{0}}\Big)(\|\overrightarrow{R}^{n+1}_{\alpha}\|_{\infty}+\|\overrightarrow{R}^{n+1}_{\Delta}\|_{\infty}).
\end{align*}
We use the identity $A^{(n)}_{0}=a^{(n)}_{0}=1/(\Gamma(2-\alpha)\tau_{n+1}^{\alpha})$ to obtain that
\begin{align*}
&\frac{\max_{\rho\in[-1,1]}|f'_{\kappa}(\rho)|}{A^{(n)}_{0}}\Big[\sum_{k=0}^{n}(A^{(n)}_{n-k}-A^{(n)}_{n-k+1})+\max_{\rho\in[-1,1]}|f'_{\kappa}(\rho)|\Big]\\
=&\frac{\max_{\rho\in[-1,1]}|f'_{\kappa}(\rho)|}{A^{(n)}_{0}}[A^{(n)}_{0}+\max_{\rho\in[-1,1]}|f'_{\kappa}(\rho)|]\\
\leq&\max_{\rho\in[-1,1]}|f'_{\kappa}(\rho)|(1+\Gamma(2-\alpha)\tau^{\alpha}\max_{\rho\in[-1,1]}|f'_{\kappa}(\rho)|)\\
:=&\Lambda
\end{align*}
for any $0\leq n\leq N-1.$
Then, it follows from Theorem \ref{thm1} that
\begin{align*}
\|e^{n+1}\|_{\infty}\leq &E_{\alpha}(\max(1,\rho)\pi_{B}\Lambda t^{\alpha}_{n+1})\Big[\|e^{0}\|_{0}+\max_{\rho\in[-1,1]}|f'_{\kappa}(\rho)|\max_{0\leq k\leq n}\sum_{j=0}^{k}\widetilde{p}^{(k)}_{k-j}\frac{\overrightarrow{R}^{j+1}_{f_{\kappa}}}{A^{(j)}_{0}}\\
&+\big(1+\Gamma(2-\alpha)\tau^{\alpha}\max_{\rho\in[-1,1]}|f'_{\kappa}(\rho)|\big)\big(\max_{0\leq k\leq n}\sum_{j=0}^{k}\widetilde{p}^{(k)}_{k-j}\|\overrightarrow{R}^{j+1}_{\alpha}\|_{\infty}\\
&+\max_{0\leq k\leq n}\sum_{j=0}^{k}\widetilde{p}^{(k)}_{k-j}\|\overrightarrow{R}^{j+1}_{\Delta}\|_{\infty}\big)\Big].
\end{align*}
Together with Lemma \ref{lem4}, Lemma \ref{lemm5}, \eqref{eqnn5} and $\ve^{0}=\bf{0}$, the desired estimates \eqref{ex1} and \eqref{ex2} can be derived, which completes the proof.
 \end{proof}

\section{Numerical results}
\label{sect4}
\setcounter{equation}{0}
In this section, we present some numerical experiments to validate the theoretical results of the proposed numerical scheme in terms of accuracy and the preservation of the discrete MBP and the discrete variation energy dissipation law.
 Throughout the numerical experiments, the central finite difference method is exploited for the spatial discretization, the fast SOE approach is used for discretization of the time-fractional derivative, and the stabilizing parameter is set to be $\kappa=2$ satisfying the condition \eqref{kappa_cd}.
\subsection{Temporal convergence}
\begin{figure}[!t]\vskip3mm
\centering
\includegraphics[scale=0.3]{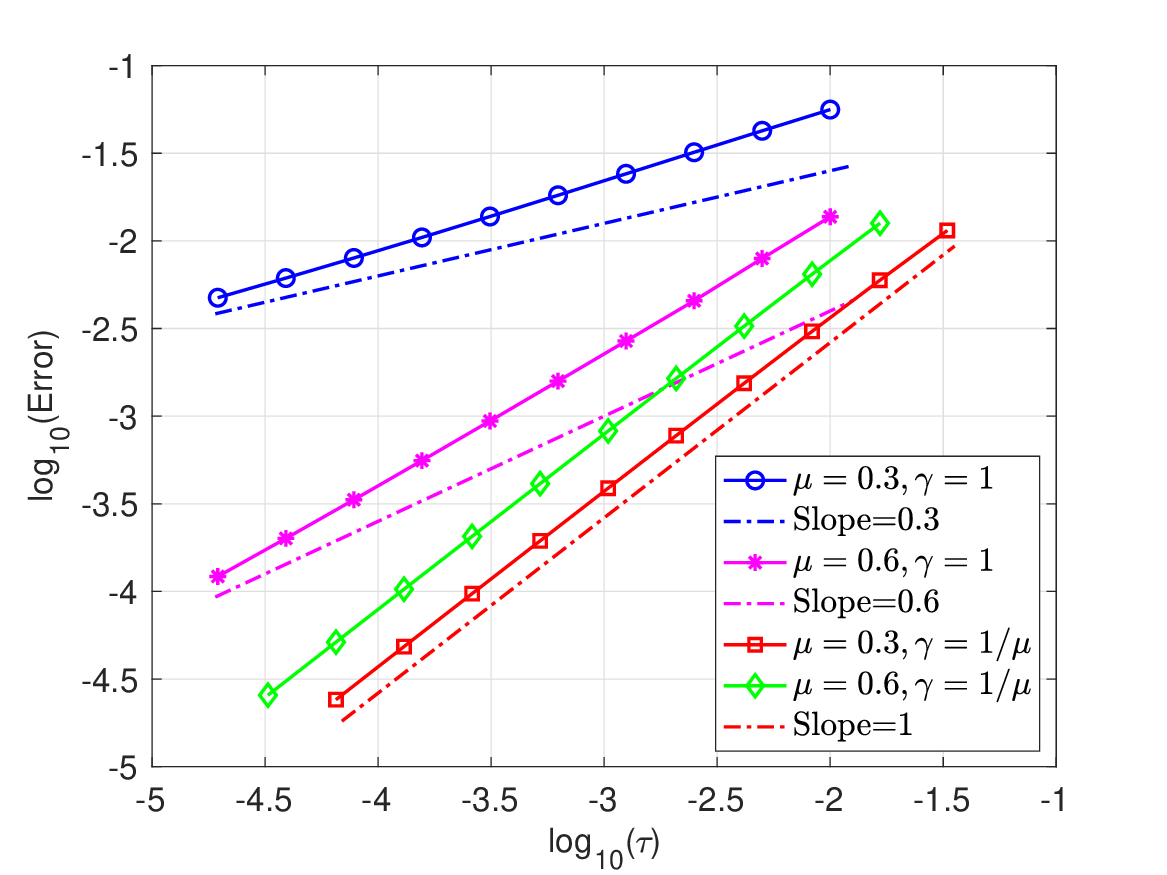}\includegraphics[scale=0.3]{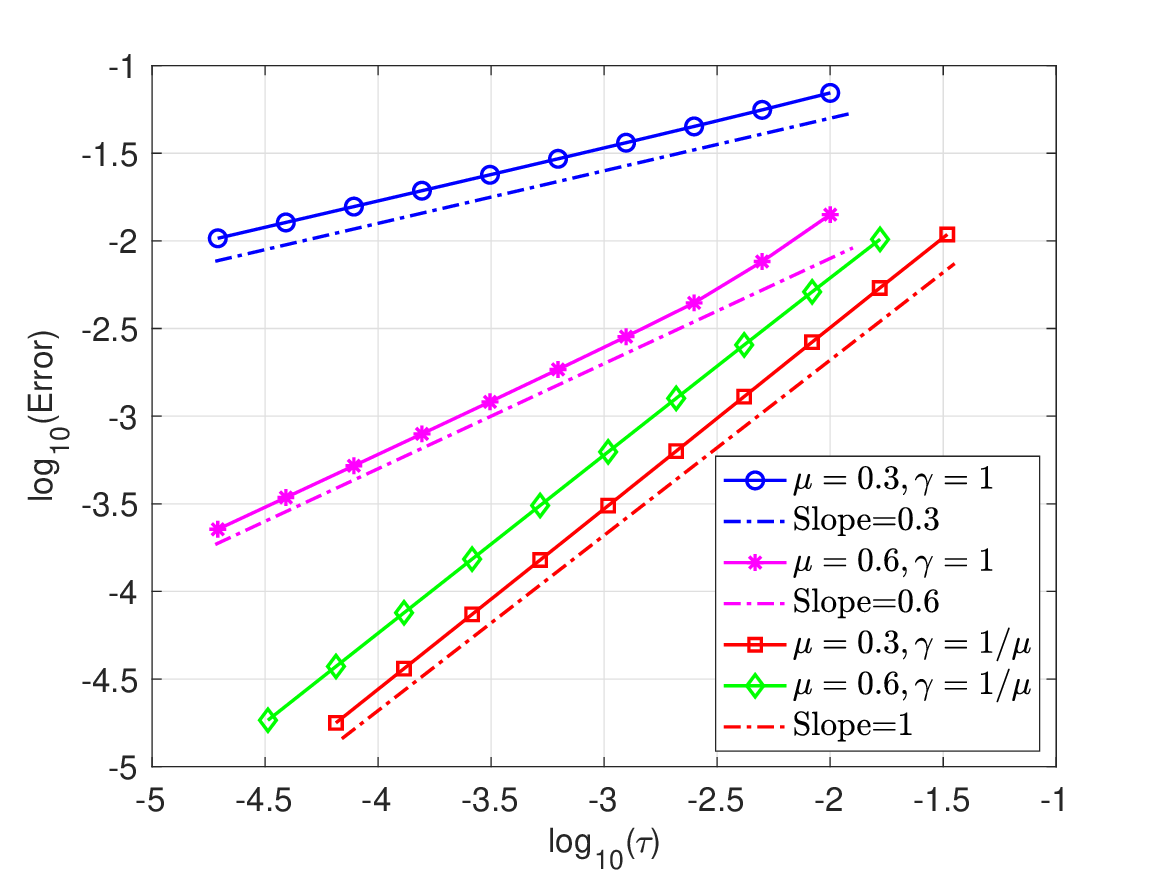}\\
\centering
{(a)  first order scheme \eqref{scheme1}. Left: $\alpha=0.3$; right: $\alpha=0.8$.}\\
\centering
\includegraphics[scale=0.3]{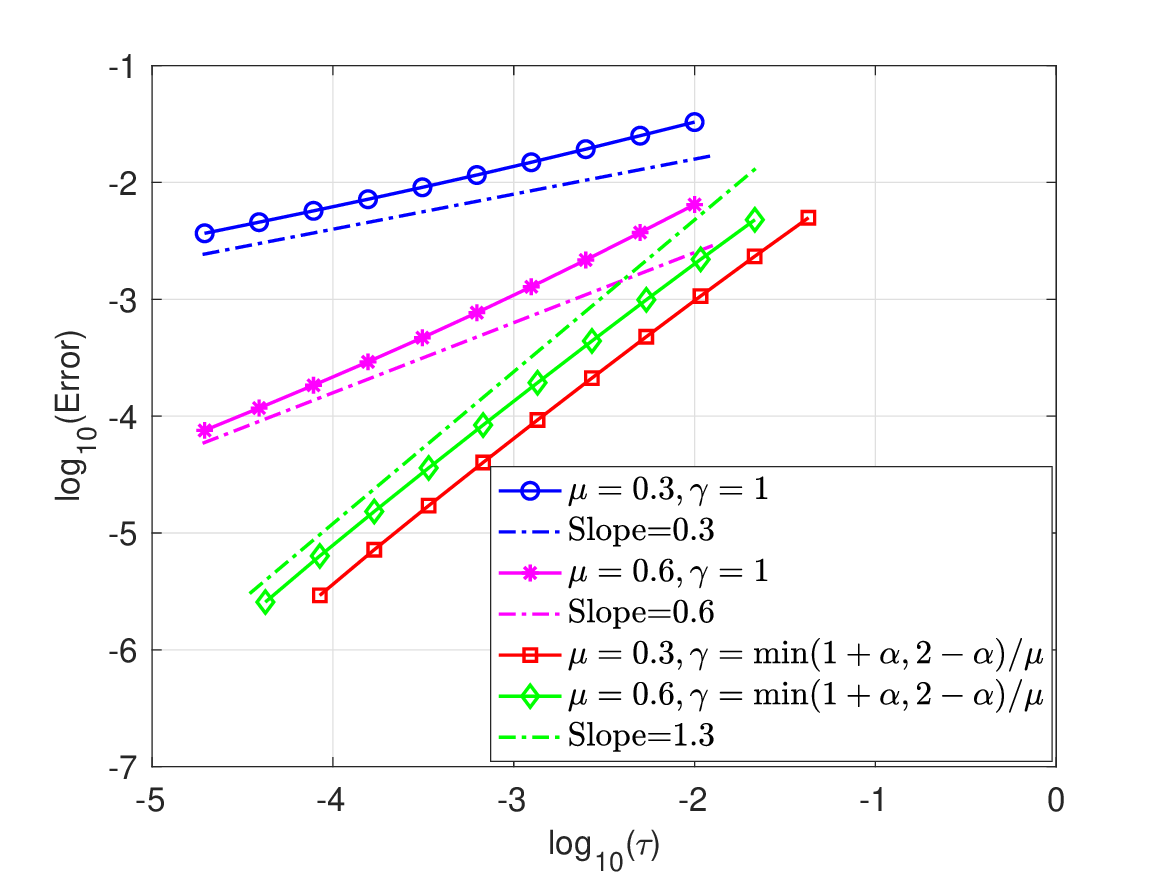}\includegraphics[scale=0.3]{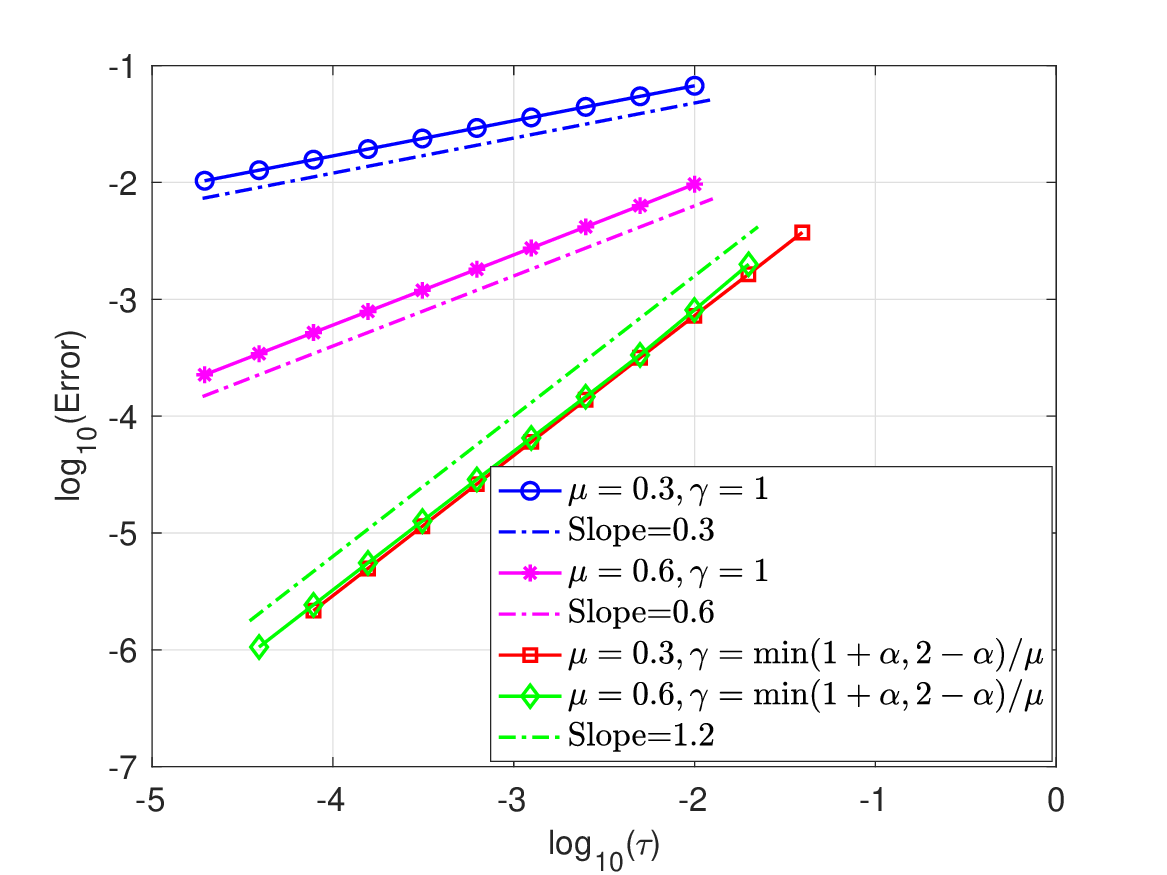}\\
\centering
{(b) predictor-corrector L1 scheme \eqref{fL1_h}. Left: $\alpha=0.3$; right: $\alpha=0.8$.}
\caption{The error behavior with respect to the time step size for the first order scheme \eqref{scheme1} and the predictor-corrector L1 scheme \eqref{fL1_h}.
}\label{fig1}
\end{figure}
We consider the following TFAC equation with the parameter $\varepsilon=0.01$ and the homogenous Neumann boundary condition:
\begin{align}
_{0}{}\!D^{\alpha}_{t}\phi-\varepsilon^{2}\Delta\phi+f(\phi)=s(\x,t),\quad (\x,t)\in \Omega\times(0,T],
\end{align}
where $s(\x, t)$ denotes a constructed source term specifically designed to ensure the exact solution takes the form:
\begin{align}
\phi(\x,t)=0.2(t^{\mu}+1)\cos(x)\cos(y).
\end{align}
This type of solutions have low regularity at the initial time.
Then, the use of uniform temporal mesh will result in a loss of convergence order,
and non-uniform grid meshes will be needed
to recover the optimal convergence order. One of non-uniform meshes,
probably the most used, is the so-called graded mesh:
\bq\label{exx1}
t_{n}=T \big({n\over N}\big)^
\gamma, \gamma\geq1, n=0,1,\cdots,N.
\eq
For the predictor-corrector L1 scheme \eqref{fL1_h}, Theorem \ref{Thm1} indicates that it will achieve $\mu$-order and the optimal $\min(1+\alpha,2-\alpha)$-order in time for the graded mesh with $\gamma=1$ and $\gamma=\min(1+\alpha,2-\alpha)/\mu$, respectively.
With similar discussion to the Theorem \ref{Thm1}, we can derive the first order scheme \eqref{scheme1} will achieve $\mu$-order and the optimal first-order in time for the graded mesh with $\gamma=1$ and $\gamma=1/\mu$, respectively.

The computational domain is set to be $\Omega = (0,1)^2$ and the terminal time is chosen to be $T=1$.
We  fix the uniform small spatial mesh size $h=2\pi/512$ and test the convergence rate in time of the proposed schemes \eqref{scheme1} and \eqref{fL1_h} on the graded mesh \eqref{exx1} with $N=100\times 2^{k},k=0,1,2,\cdots,9.$
In Figure \ref{fig1}, we present the $L^{\infty}$ errors at $T=1$ as functions of the time step sizes in log-log scale.
It is shown in Figure \ref{fig1} (a) that for the tested $\alpha=0.3$ and $\alpha=0.8$ the scheme \eqref{scheme1} achieves the expected $\mu$-order and first-order temporal accuracy for the graded mesh with $\gamma=1$ and $\gamma=1/\mu$, respectively.
Once again the observed error behaviors in Figure \ref{fig1} (b) achieve the desired $\mu$-order and $\min(1+\alpha,2-\alpha)$-order accuracy in time for the graded mesh with $\gamma=1$ and $\gamma=\min(1+\alpha,2-\alpha)/\mu$, respectively.

\subsection{The grain coasening dynamics}
In this numerical experiment, we examine the coarsening dynamics governed by the TFAC equation \eqref{TFAC} with a random initial condition uniformly distributed in the range $[-0.001, 0.001]$ and the width parameter $\varepsilon^{2} = 10^{-3}$.

We investigate the unconditional MBP preservation and the discrete variational energy dissipation of the proposed fast linear predictor--corrector stabilized L1 scheme \eqref{fL1_h} through a long-time simulation up to $T = 500$, using several large time step sizes.
The computational domain is set as $\Omega = (-1,1)^{2}$ and a uniform spatial mesh with grids $128 \times128$ is employed for spatial discretization.
The simulations are conducted on a graded temporal mesh with grading parameter $\gamma = \min(1+\alpha, 2-\alpha) / \alpha$ and $N = 100$ in $[0, 0.1]$, followed by a uniform temporal mesh with various time-step sizes in $(0.01, T]$.
The time evolutions of the supremum norm and the discrete variational energy for the tested cases are illustrated in Figure \ref{fig2}. The results demonstrate that the proposed scheme \eqref{fL1_h} preserves the MBP and ensures discrete variational energy dissipation under all tested time-step sizes.
\begin{figure}[!t]
\centering
\includegraphics[scale=0.3]{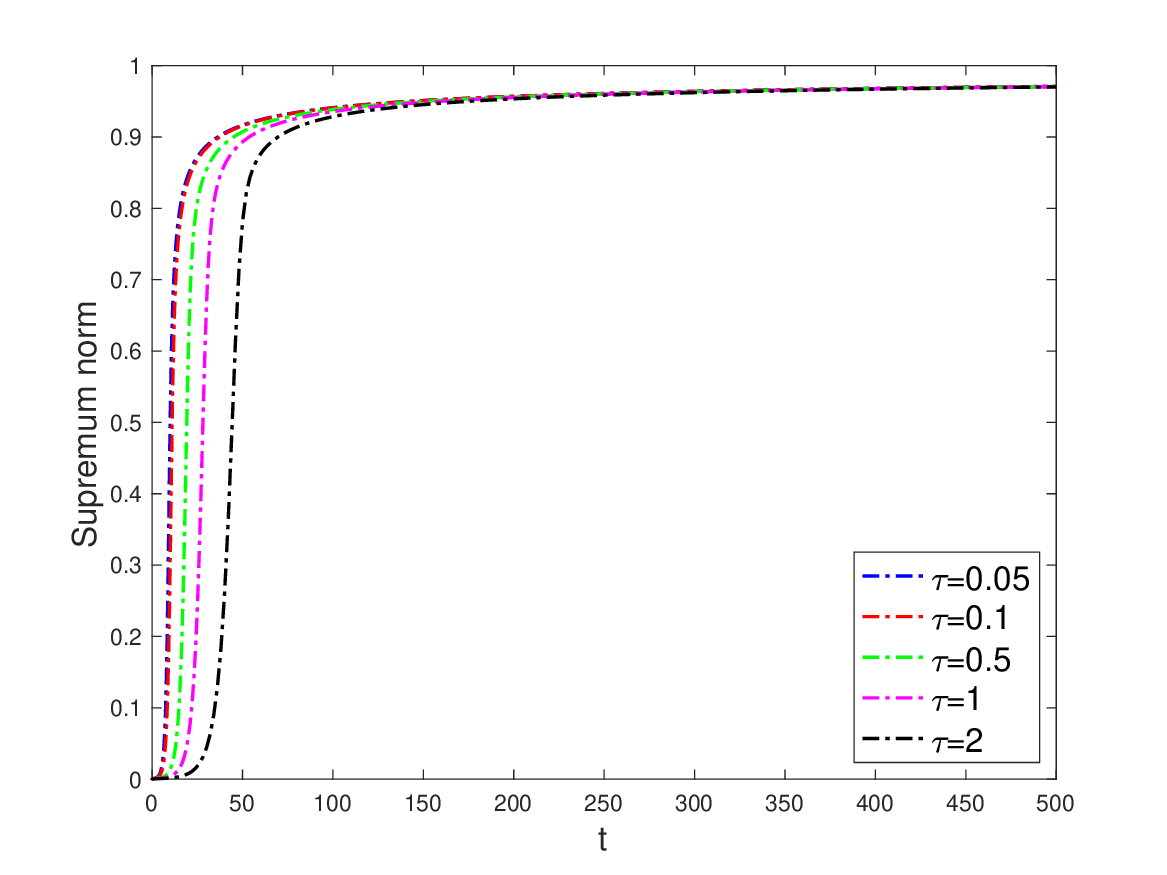}\includegraphics[scale=0.3]{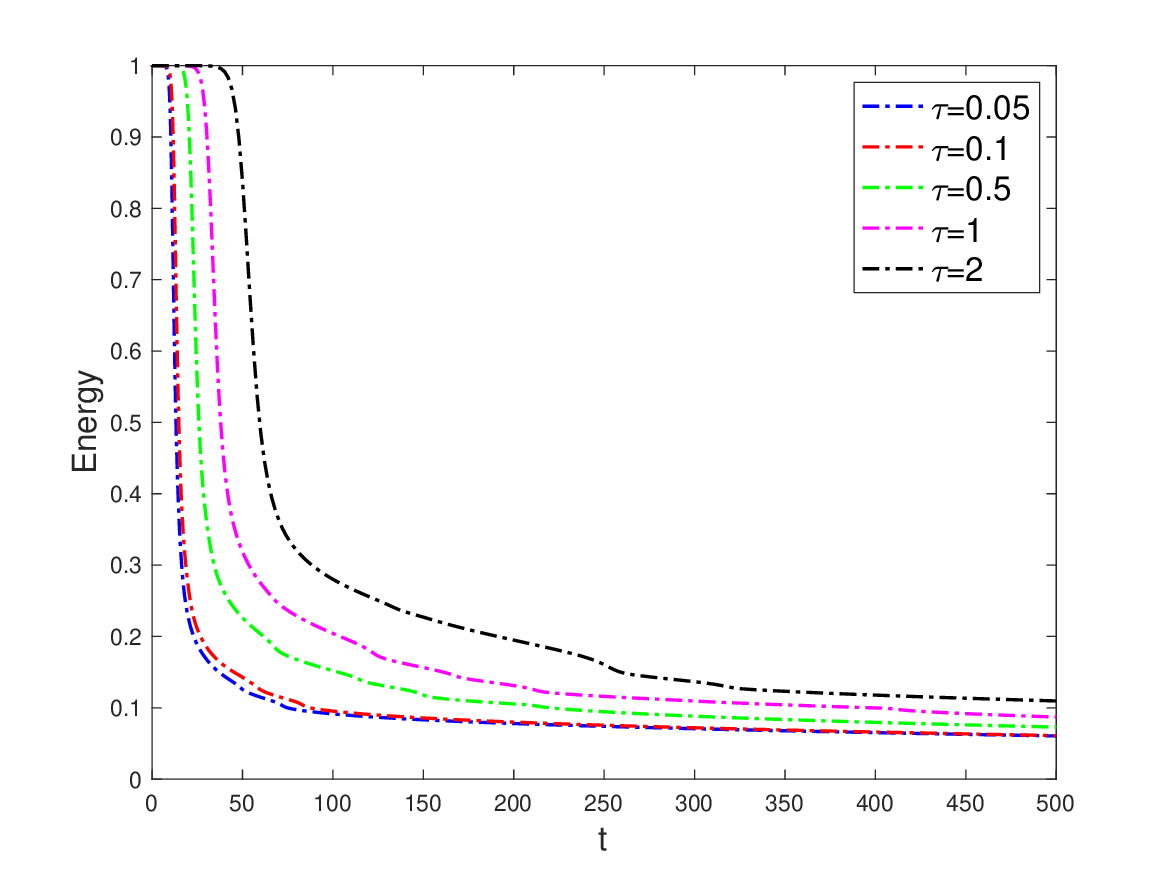}\\
\centering
{(a) $\alpha=0.4$  }\\
\centering
\includegraphics[scale=0.3]{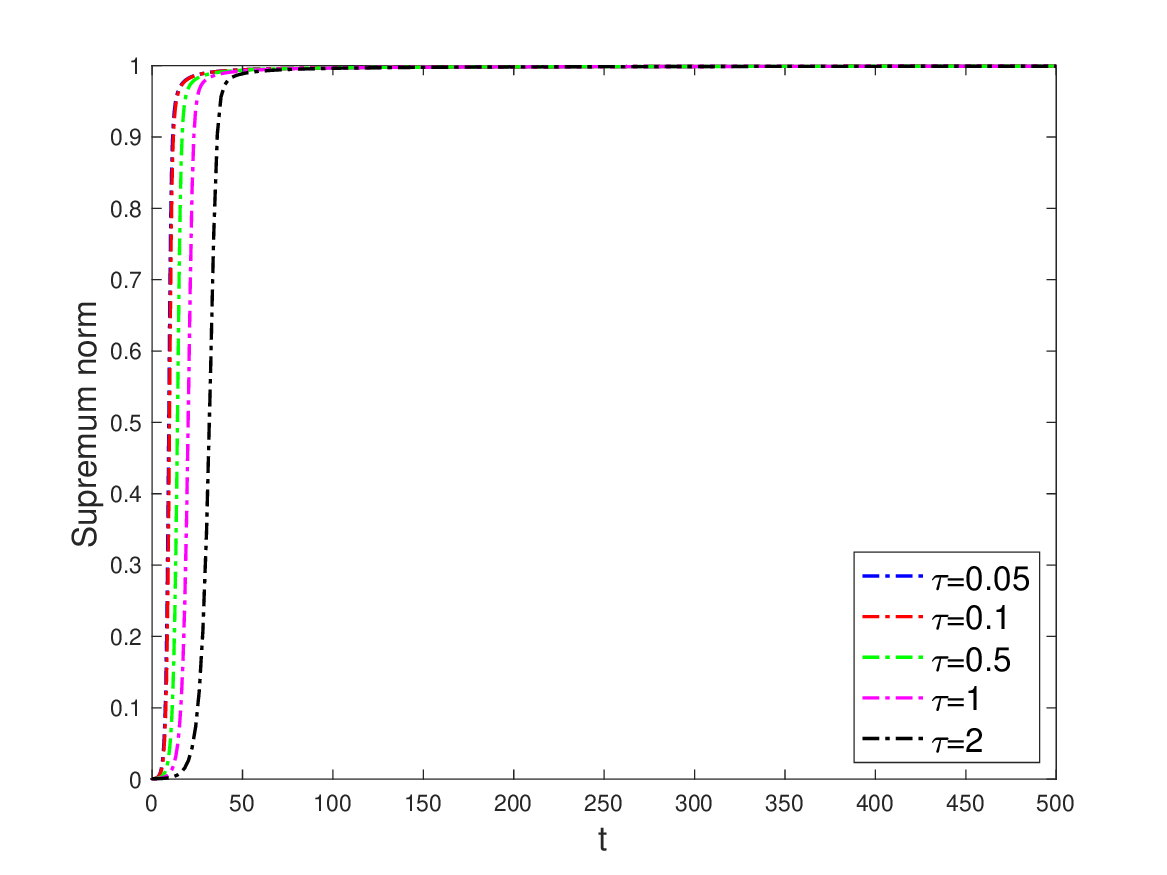}\includegraphics[scale=0.3]{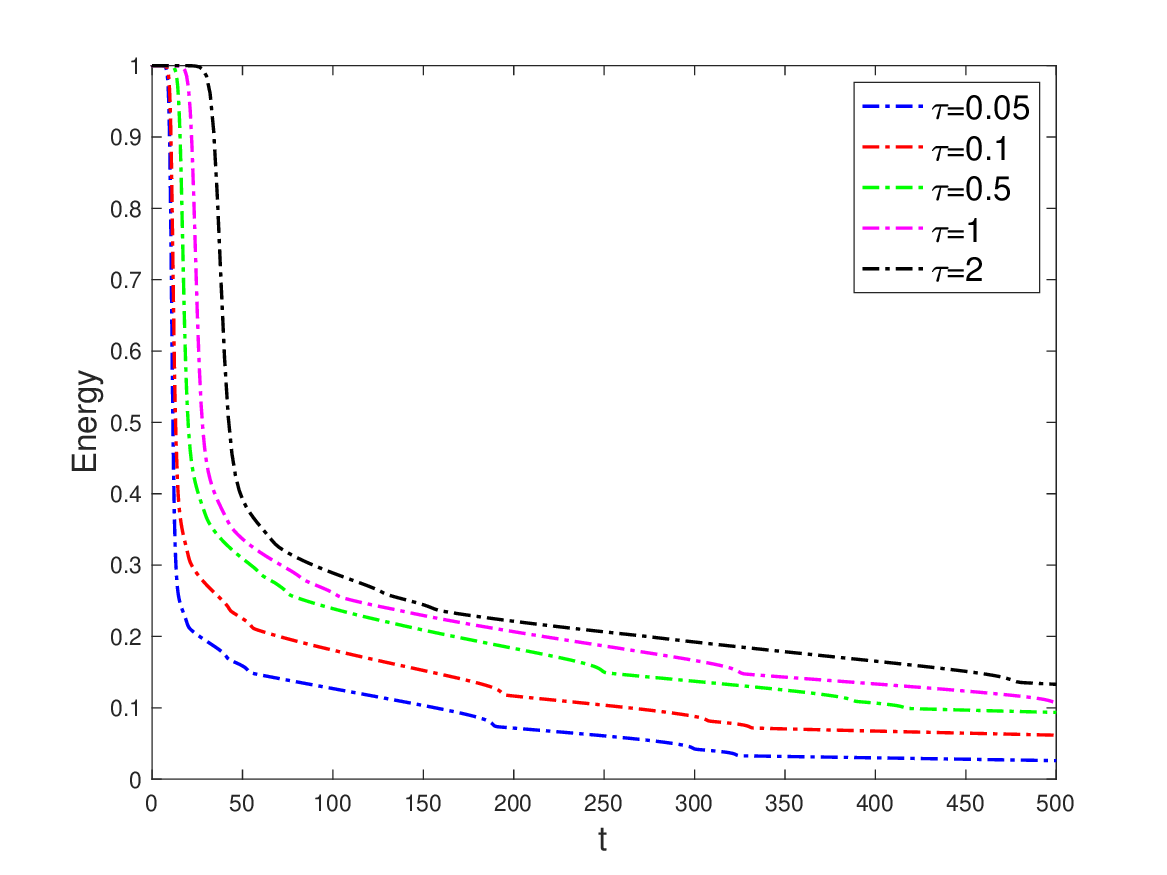}\\
\centering
(b) $\alpha=0.8$
\caption{The evolutions in time of the supremum norm (left) and the variant energy (right)  of the simulated solution produced by the  proposed fast linear predictor--corrector stabilized L1 scheme \eqref{fL1_h} with some uniform time steps for the grain coarsening  problem.}\label{fig2}
\end{figure}
\begin{figure*}[t!]
\centerline{\includegraphics[scale=0.28]{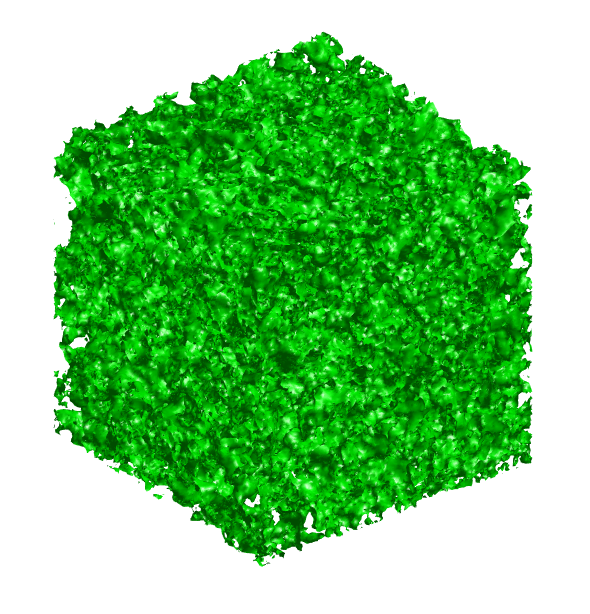}\includegraphics[scale=0.28]{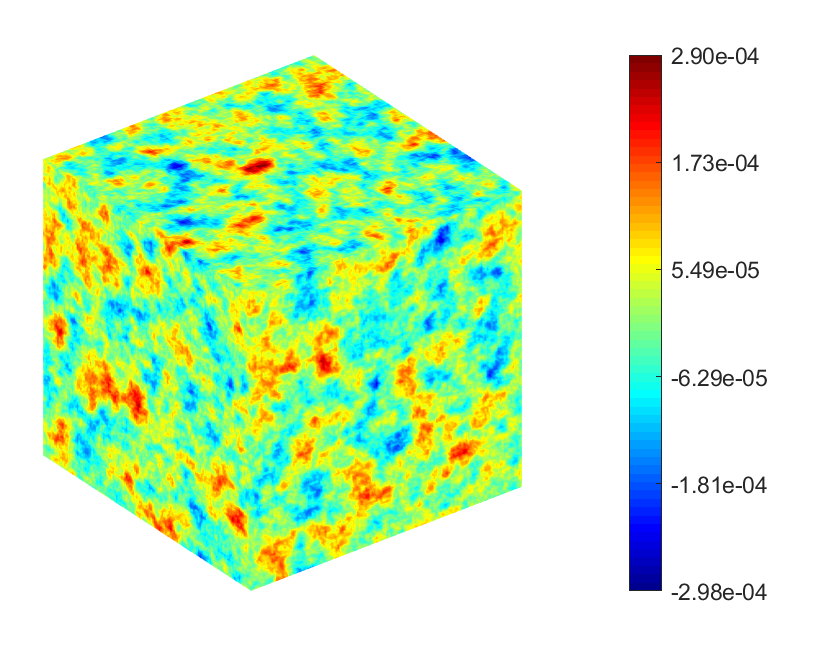}\includegraphics[scale=0.28]{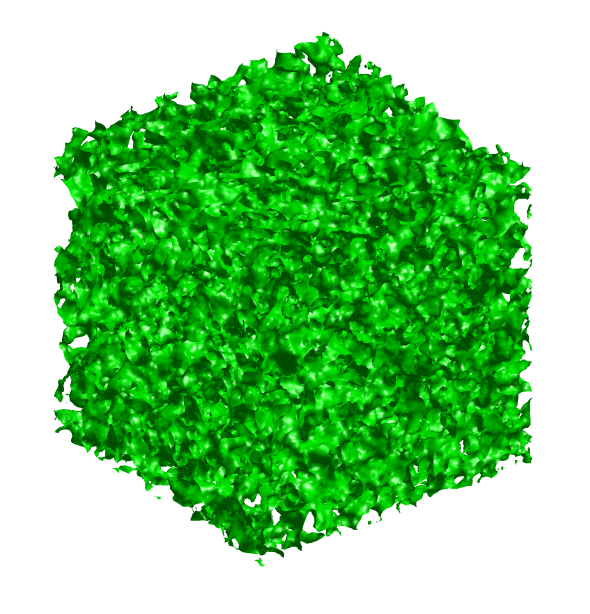}\includegraphics[scale=0.28]{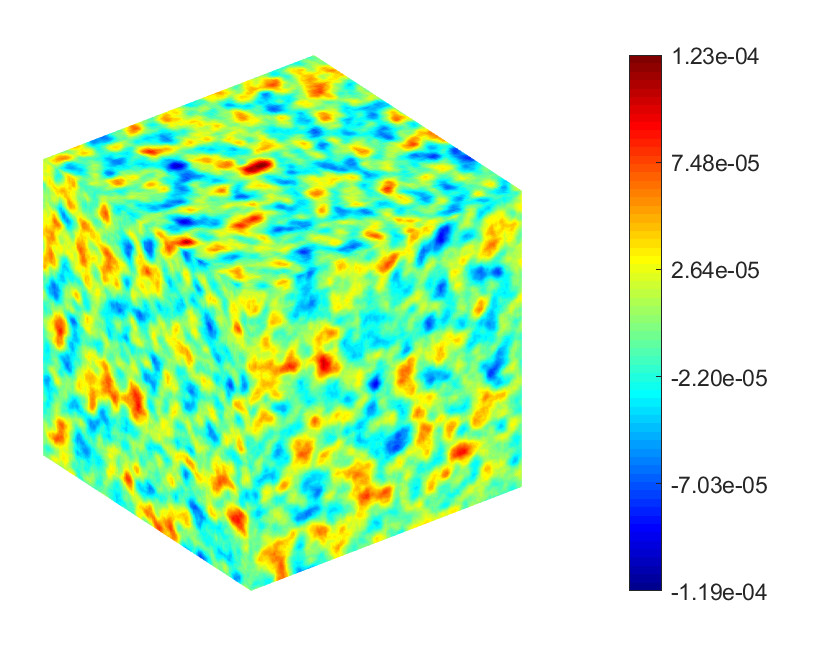}}
\centerline{\includegraphics[scale=0.28]{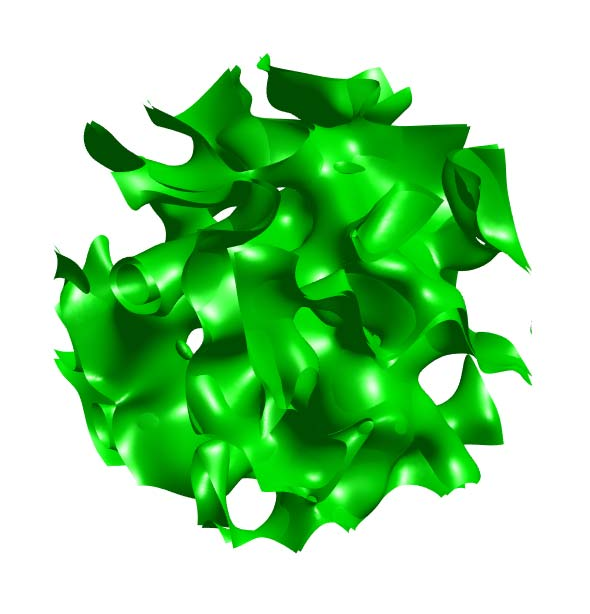}\includegraphics[scale=0.28]{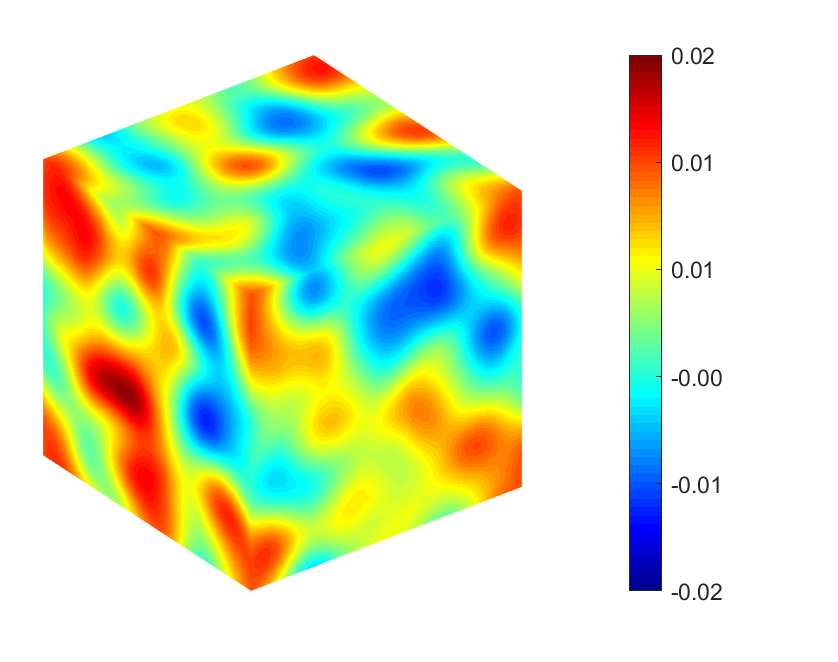}\includegraphics[scale=0.28]{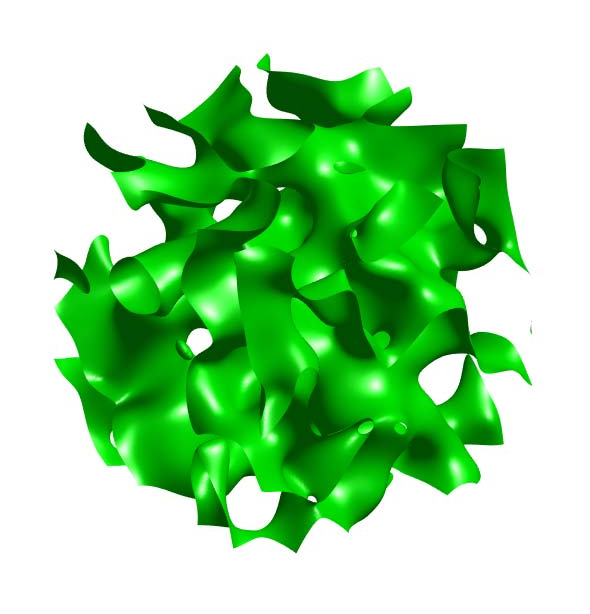}\includegraphics[scale=0.28]{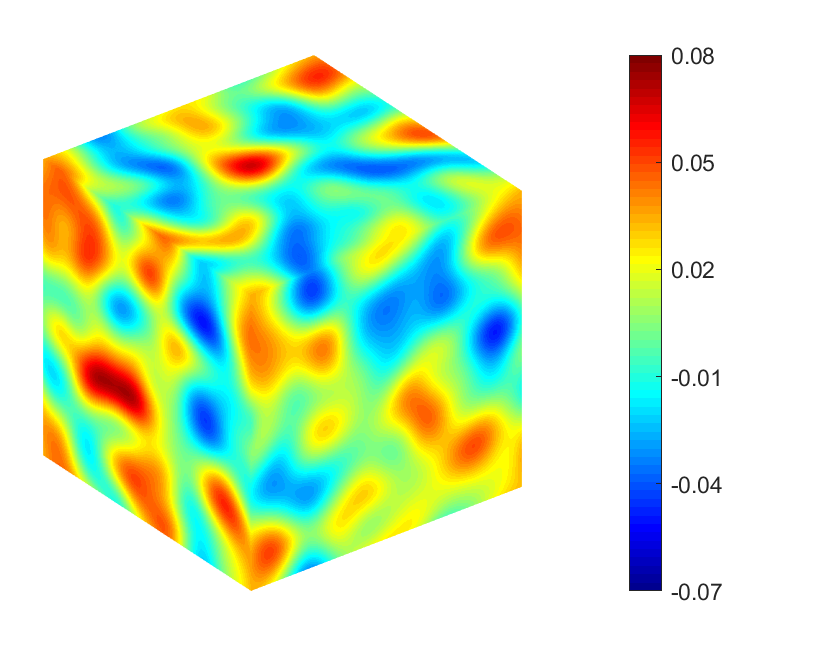}}
\centerline{\includegraphics[scale=0.28]{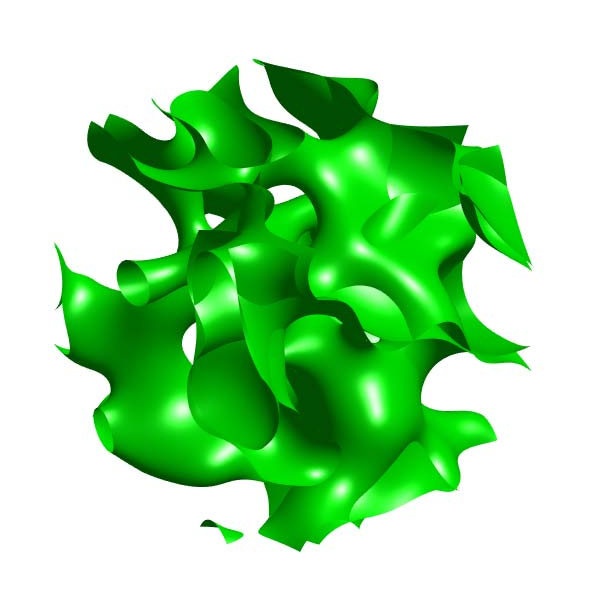}\includegraphics[scale=0.28]{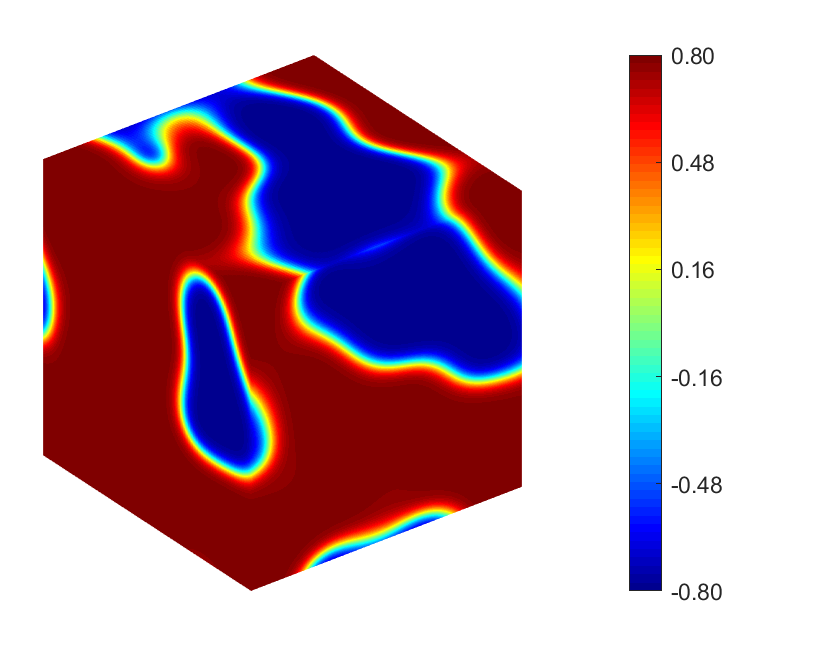}\includegraphics[scale=0.28]{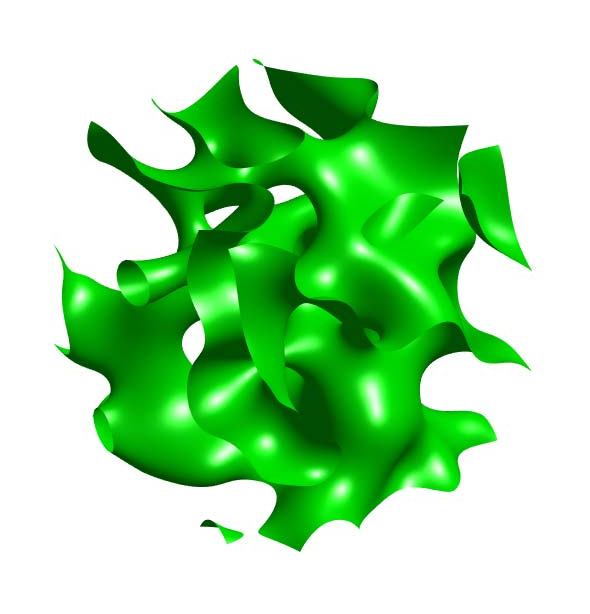}\includegraphics[scale=0.28]{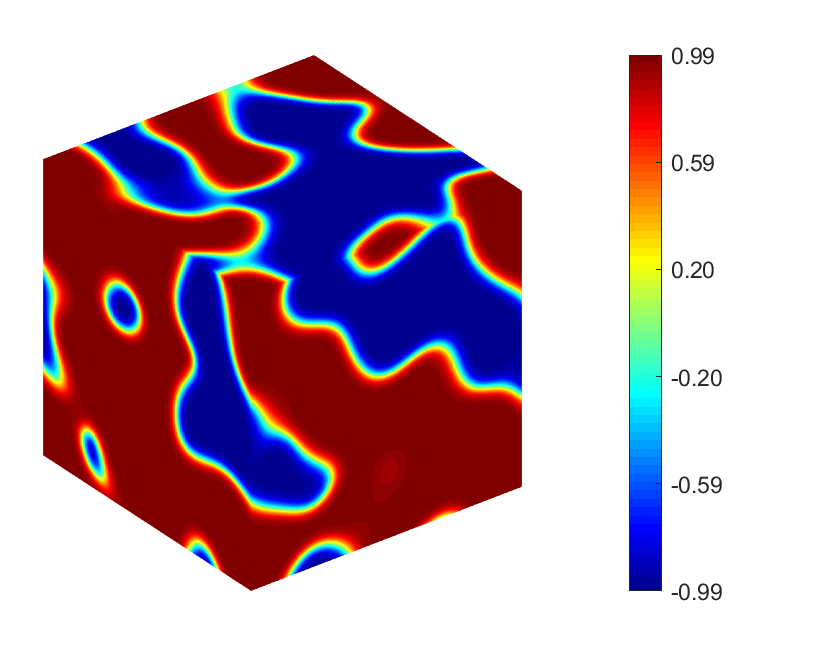}}
\centerline{\includegraphics[scale=0.28]{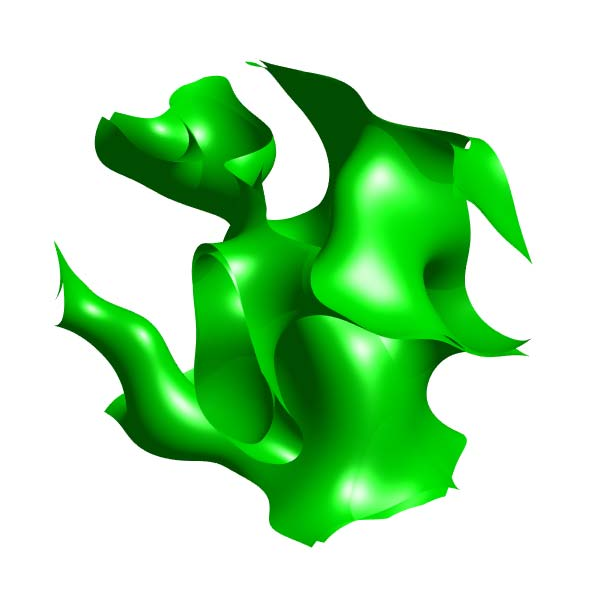}\includegraphics[scale=0.28]{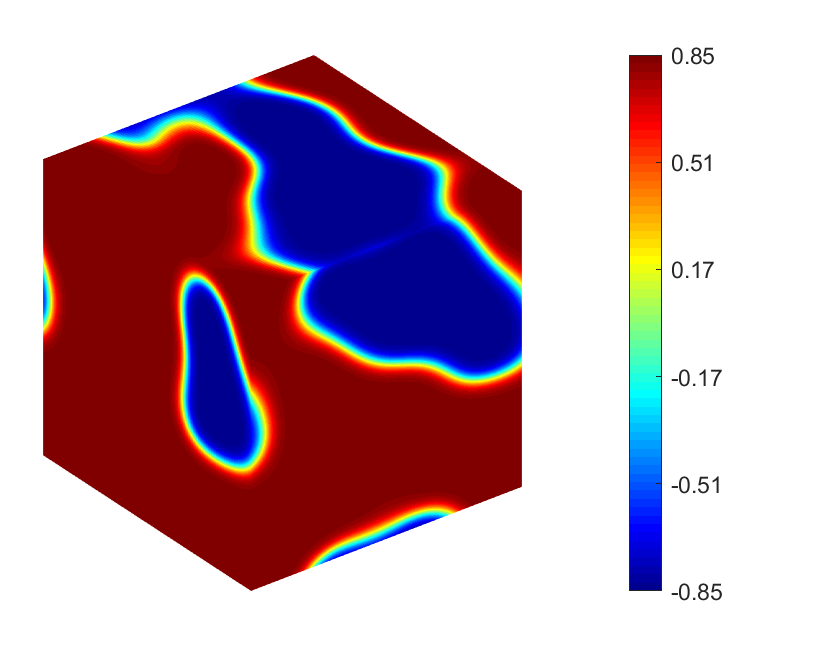}\includegraphics[scale=0.28]{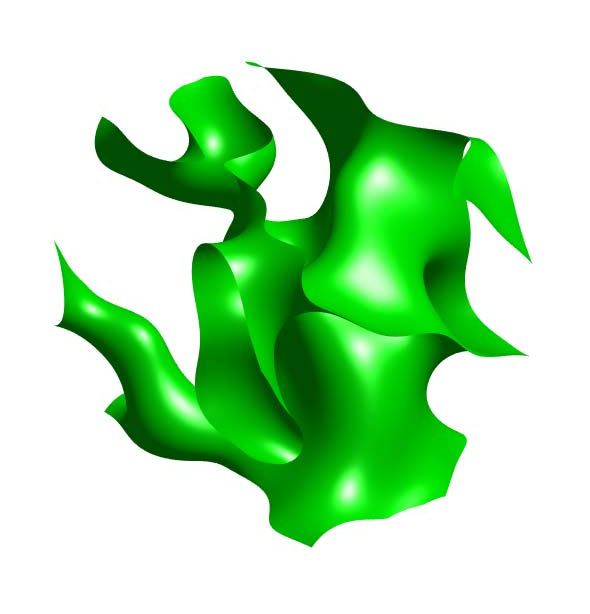}\includegraphics[scale=0.28]{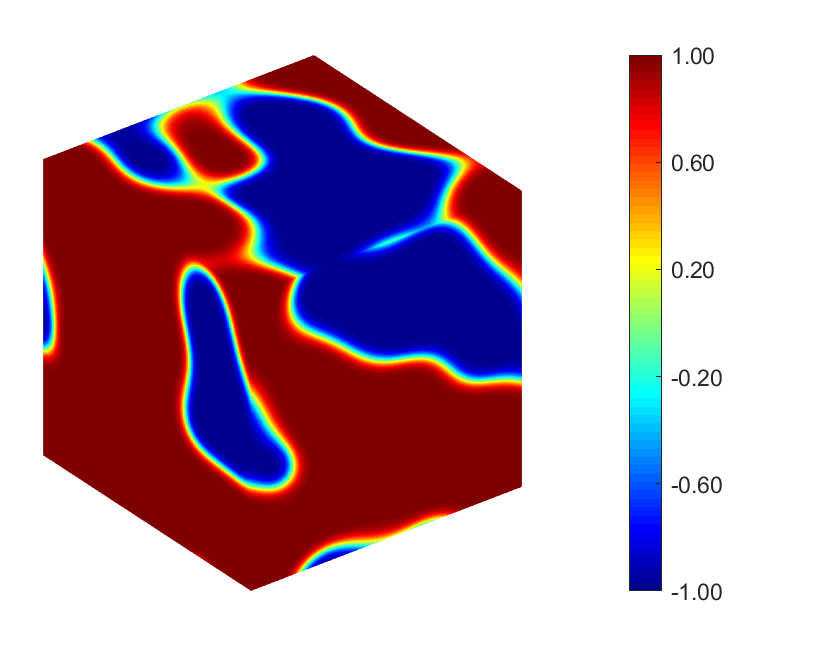}}
\centerline{\includegraphics[scale=0.28]{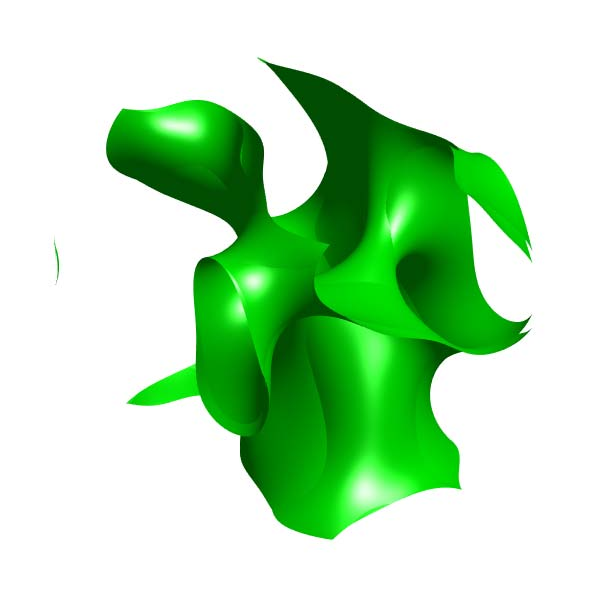}\includegraphics[scale=0.28]{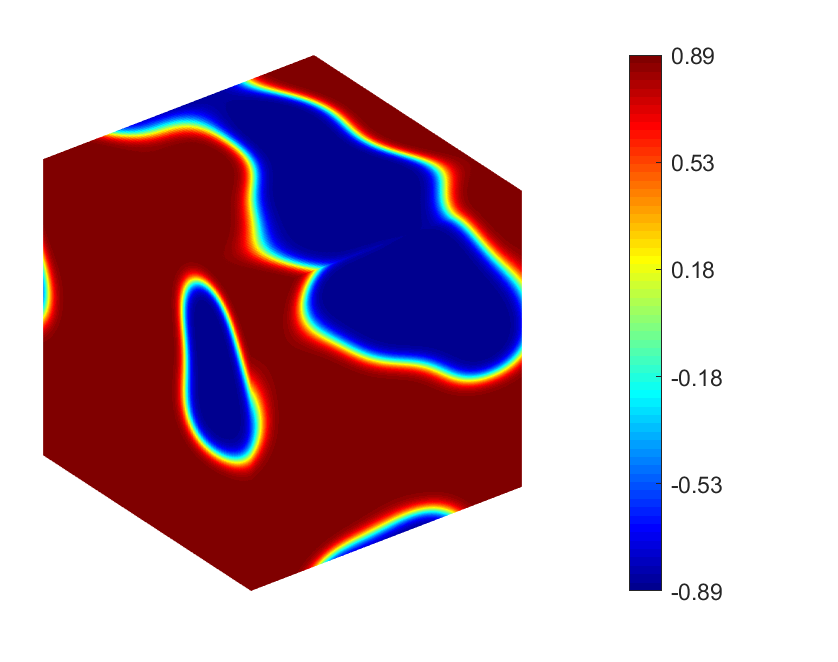}\includegraphics[scale=0.28]{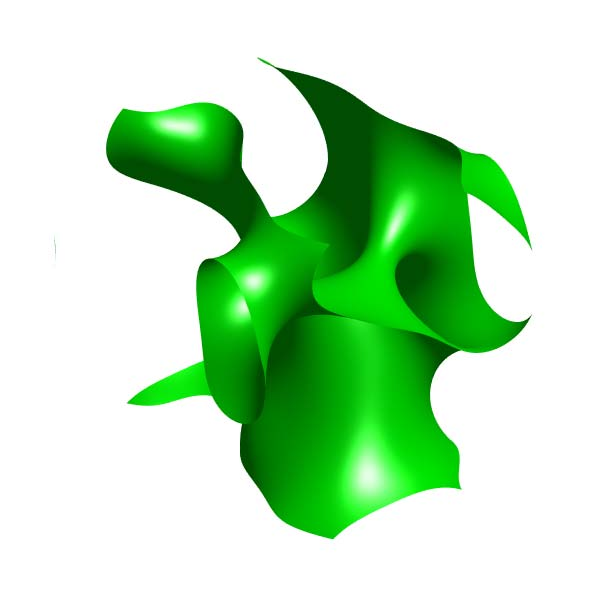}\includegraphics[scale=0.28]{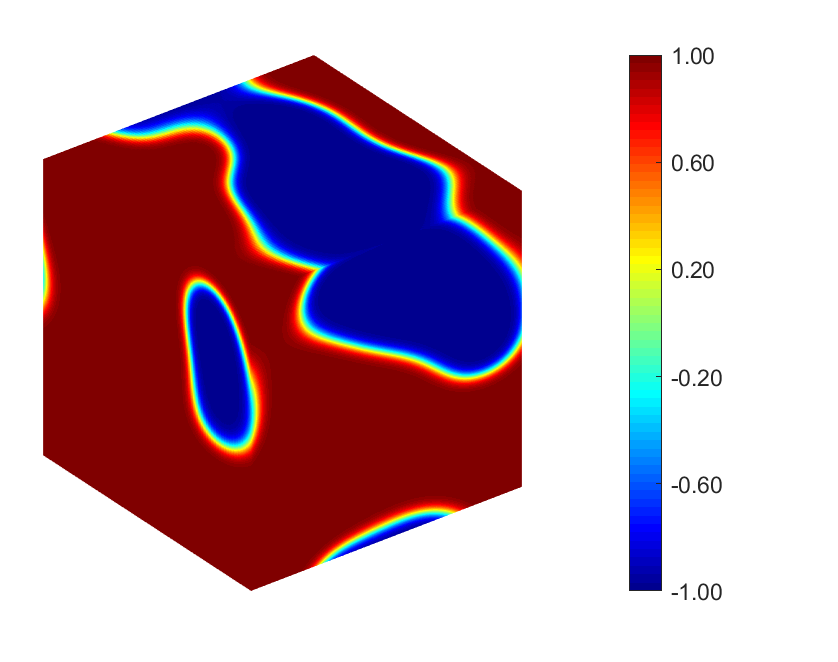}}
\caption{Snapshots of the iso-surfaces of the phase function $\phi=0$ (the first and third columns) and  density field $\phi$ (the second and last columns) for the 3D model \eqref{TFAC} with $\alpha=0.3$ (the first two columns) and $\alpha=0.9$ (the last two columns) around the times $t=1$, $10$, $20$, $50$ and $100$.}
\label{fig3}
\end{figure*}
Unconditionally stable schemes offer a key advantage in their compatibility with adaptive time-stepping strategies. This feature is particularly valuable for long-time simulations of coarsening dynamics, where phase separation typically proceeds through several distinct stages over extended timescales: rapid changes in the early stage, followed by a markedly slower evolution as the system approaches a steady state (see Figure \ref{fig2}). A variety of efficient time-adaptive strategies has been developed \cite{HQ22, HQ23, LTZ20, QZT11, STY16, Shen17_2}, and these can be integrated with numerical schemes that employ variable time steps. In particular, a robust energy-variation-based adaptive strategy was proposed in~\cite{QZT11} to efficiently capture the coarsening dynamics:
\begin{equation}\label{adp}
\Delta t_{n+1} = \max\left(\Delta t_{\min}, \frac{\Delta t_{\max}}{\sqrt{1 + \alpha |\mathcal{E}'{\alpha}(t)|^{2}}}\right),
\end{equation}
where $\Delta t_{\min}$ and $\Delta t_{\max}$ denote the prescribed minimum and maximum time-step sizes, respectively, and $\alpha$ is a positive constant.
This adaptive mechanism balances accuracy and efficiency by assigning smaller time steps during periods of rapid energy dissipation and larger ones when the system evolves more smoothly.

In what follows, we use the adaptive strategy in conjunction with the predictor-corrector stabilized L1 scheme~\eqref{fL1_h} to simulate the coarsening dynamics of the three-dimensional model~\eqref{TFAC} with different time-fractional orders $\alpha$.
The computational domain is chosen as $(-1,1)^{3}$, discretized using a uniform grid of $128\times128\times128$ points.
The parameters of the adaptive strategy~\eqref{adp} are set to $\tau_{\min}=10^{-5}$, $\tau_{\max}=0.1$, and $\alpha=10^{5}$ in this test.
Figure~\ref{fig3} presents snapshots of iso-surfaces of the phase function $\phi=0$ together with the corresponding density field $\phi$, where the colorbar ranges from $\min\{\phi^{n}\}$ to $\max\{\phi^{n}\}$, at different times for fractional orders $\alpha=0.4$ and $\alpha=0.9$.
The results show that, at the early stage, the coarsening dynamics for $\alpha=0.4$ proceeds faster than that for $\alpha=0.9$, whereas at later times the growth slows down significantly as the fractional order decreases.
This behavior is consistent with the two-dimensional case reported in~\cite{HX21,DYZ20}, and the corresponding evolutions of the computed supremum and variant energy are displayed in Figure~\ref{fig4}, highlighting a phenomenon whose underlying mechanism remains unclear and warrants further investigation.
\begin{figure}[!t]
\centering
\includegraphics[scale=0.3]{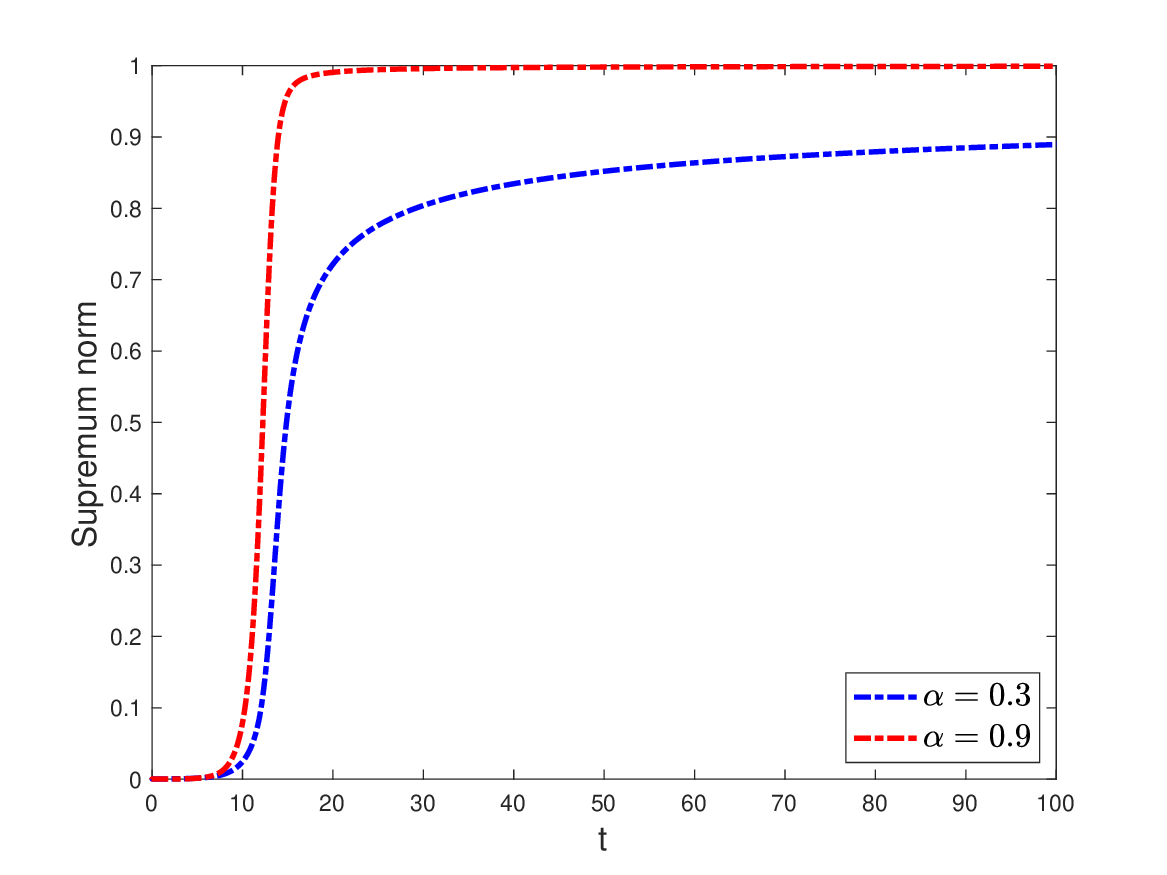}\includegraphics[scale=0.3]{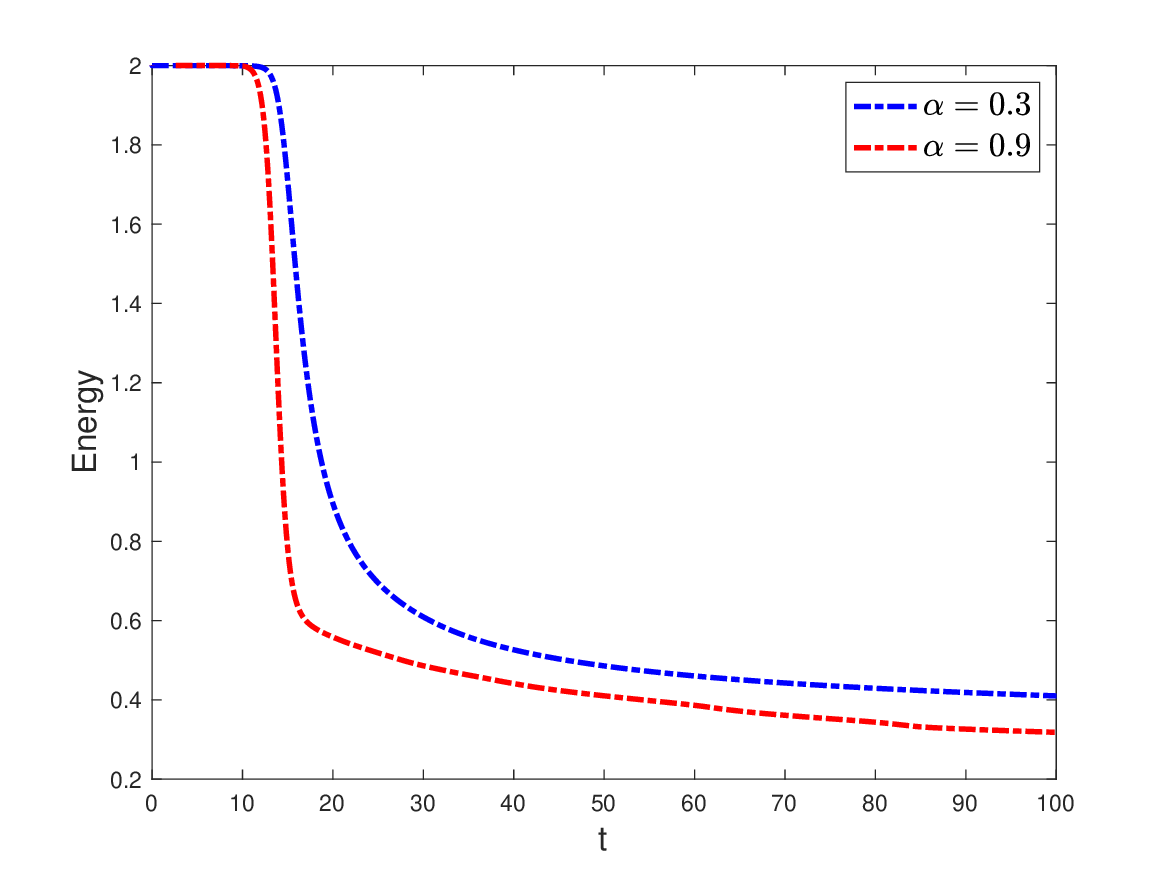}\\
\caption{The evolutions in time of the supremum norm (left) and the energy (right)  of the simulated solution produced by the fast linear predictor--corrector stabilized L1 scheme \eqref{fL1_h}.}\label{fig4}
\end{figure}

\section{Concluding remarks}
We have developed several unconditionally MBP-preserving and energy-dissipation–stable linear stabilized L1 schemes for the TFAC equation. The resulting fully discrete system uses central finite differences in space and requires solving one or two Poisson-type problems at each time step.
The predictor--corrector method and a stabilizing term are introduced to unconditionally preserve the discrete MBP  and the variant energy dissipation of the proposed schemes.
We rigorously establish the discrete preservation of the MBP and the variant energy dissipation law \eqref{Eg_law} of the proposed schemes.
A extend fraction Gr\"onwall inequality with $\theta=1$ is derived and the discrete $L^{\infty}$ error estimates are rigorously derived using the this type fraction Gr\"onwall inequality which plays a crucial role in eliminating the time-step size restriction in the error analysis of the proposed linear schemes.
Finally, a series of numerical experiments were carried out to verify the theoretical claims and illustrate the efficiency of the proposed fast linear predictor--corrector stabilized L1 scheme with a time adaptive strategy.
It remains of interest to further investigate the construction of high-order linear predictor-corrector stabilized schemes based on the L2 and L2-$1_{\sigma}$ approximations of the time-fractional derivative in future work.

\bibliographystyle{siamplain}
\bibliography{ref}

\end{document}

%% file: ex_shared.tex
\usepackage{lipsum}
\usepackage{amsfonts}
\usepackage{graphicx}
\usepackage{epstopdf}
\usepackage{algorithmic}
\usepackage{amssymb}
\ifpdf
  \DeclareGraphicsExtensions{.eps,.pdf,.png,.jpg}
\else
  \DeclareGraphicsExtensions{.eps}
\fi

\newsiamremark{remark}{Remark}
\newsiamremark{hypothesis}{Hypothesis}
\crefname{hypothesis}{Hypothesis}{Hypotheses}
\newsiamthm{claim}{Claim}

\headers{Structure-preserving L1 scheme for TFAC equation}{D. Hou, Z. Qiao and T. Tang}

\title{A linear unconditionally structure-preserving L1 scheme for the time-fractional Allen-Cahn equation\thanks{
Submitted to the editors DATE.
\funding{D. Hou is partially supported by Natural Science Foundation of China grant 12001248 and 12571387, and the NSF of Jiangsu Province grant
BK20201020.
Z. Qiao is partially supported by Hong Kong Research Grants Council RFS Grant RFS2021-5S03 and GRF Grants 15302122 and 15305624, and CAS AMSS-PolyU Joint Laboratory of Applied Mathematics (JLFS/P-501/24).
 T. Tang is partially supported by the Guangdong Provincial Key Laboratory of Interdisciplinary Research and Application for Data Science under UIC 2022B12120100066, and by Science Challenge Project (Grant TZ2018001) and NSFC 11731006 and K20911001.
}}}

\author{Dianming Hou\thanks{School of Mathematics and Statistics, Jiangsu Normal University, Xuzhou, Jiangsu 221116, China.
  (\email{dmhou@jsnu.edu.cn}).}
\and Zhonghua Qiao\thanks{Corresponding author. Department of Applied Mathematics, The Hong Kong Polytechnic University, Hung Hom, Kowloon, Hong Kong.
  (\email{zqiao@polyu.edu.hk}).}
  \and Tao Tang\thanks{School of Mathematics and Statistics, Guangzhou Nanfang College, Guangzhou, Guangdong Province, China.
  (\email{ttang@nfu.edu.cn}).}
}

\usepackage{amsopn}
